\documentclass[12pt]{amsart}
\usepackage{geometry}                % See geometry.pdf to learn the layout options. There are lots.
\usepackage{pdftricks}
\begin{psinputs}
\usepackage{pstricks}
\usepackage{multido}
\end{psinputs}
\usepackage{graphicx}
\usepackage{amssymb}
\usepackage{amsthm}
\usepackage{epstopdf}
\usepackage{enumerate}
\usepackage{pst-all}
\newtheorem{theorem}{Theorem}[section]
\newtheorem{lemma}[theorem]{Lemma} 
\newtheorem{remark}[theorem]{Remark} 
\newtheorem{definition}[theorem]{Definition} 
\newtheorem{corollary}[theorem]{Corollary}

\def\da{\downarrow}

\def\A{\mathcal{A}}

\def\K{\mathcal{K}}
\def\M{\mathbb{M}}
\def\a{\alpha}

\def\o{\omega}

\title{On model theory of covers of algebraically closed fields}

\begin{document}
  
\thanks{Research of the first author was partially supported by grant 40734 of the Academy of Finland}
\thanks{Research of the second author was supported by Finnish National Doctoral Programme in Mathematics and its Applications}

 \maketitle

\begin{center}
Tapani Hyttinen \\
Department of Mathematics, University of Helsinki \\
P.O. Box 68, 00014, Finland \\
tapani.hyttinen@helsinki.fi\\
\vspace{5 mm}
 Kaisa Kangas \\
Department of Mathematics, University of Helsinki \\
P.O. Box 68, 00014, Finland \\ 
kaisa.kangas@helsinki.fi \\
\vspace{5 mm}
\end{center}

\begin{abstract}
We study covers of the multiplicative group of an algebraically closed field as quasiminimal pregeometry structures and prove that they satisfy the axioms for Zariski-like structures presented in  \cite{lisuriart}, section 4. 
These axioms are intended to generalize the concept of a Zariski geometry into a non-elementary context.
In the axiomatization, it is required that for a structure $\M$, there is, for each $n$, a collection of subsets of $\M^n$, that we call the \emph{irreducible sets}, satisfying certain properties.
These conditions are generalizations of some qualities of irreducible closed sets in the Zariski geometry context.
They state that some basic properties of closed sets (in the Zariski geometry context) are satisfied and that specializations behave nicely enough. They also ensure that there are some traces of Compactness even though we are working in a non-elementary context. 
\end{abstract}

\vspace{5 mm}
\begin{center}
\end{center}  
  
\section{Introduction}

We will be studying covers of the multiplicative group of an algebraically closed field, which are algebraic structures defined as follows.
 
\begin{definition}
Let $V$ be a vector space over $\mathbb{Q}$ and let $F$ be an algebraically closed field of characteristic $0$.
A \emph{cover} of the multiplicative group of $F$ is a structure represented by an exact sequence 
\begin{displaymath}
0 \to K \to V \to F^{*} \to 1,
\end{displaymath}
where the map $V \to F^{*}$ is a surjective group homomorphism from $(V,+)$ onto $(F^{*},\cdot)$ with kernel $K$.
We will call this map exp.
\end{definition}

We will consider a cover as a structure $V$ in the language $\mathcal{L}=\{0, +, f_q, R_{+}, R_0\}_{q \in \mathbb{Q}}$, where $V$ consists of the elements in the vector space, $0$ is a constant symbol denoting the zero element of the vector space $V$, $+$ is a binary function symbol denoting addition on $V$, and for each $q \in \mathbb{Q}$, $f_q$ is a unary function symbol denoting scalar multiplication by the number $q$.
The symbol $R_+$ is a ternary relation symbol interpreted so that $R_+(v_1, v_2, v_3)$ if and only if $\textrm{exp}(v_1)+\textrm{exp}(v_2)=\textrm{exp}(v_3)$, and $R_0$ is a binary relation symbol interpreted so that $R_0(v_1, v_2)$ if and only if $\textrm{exp}(v_1)+\textrm{exp}(v_2)=0$.
Note that field multiplication is definable using vector space addition.
 
However, for the sake of readability, we will be using the concepts of a vector space $V$ (called the \emph{cover}) and a field $F$ together with the usual algebraic notation when expressing statements about the structure.
If $v=(v_1, \ldots, v_n) \in V^n$, we write $\textrm{exp}(v)$ for $(\textrm{exp}(v_1), \ldots, \textrm{exp}(v_n)) \in F^n$.
For a set $X \subseteq F$, we write $\textrm{log }X$ for the set $\{v \in V \, | \, \textrm{exp}(v) \in X\}$.

Covers have been previously studied in the context of model theory by B. Zilber (\cite{Zilber}) and by L. Smith in her Ph.D. thesis (\cite{Lucy}).
The first-order theory of cover structures is complete, submodel complete, superstable and admits elimination of quantifiers (\cite{ZilUusi}, \cite{Zilber}).
Moreover, with an additional axiom (in $L_{\omega_1 \omega}$) stating $K \cong \mathbb{Z}$, the class is categorical in uncountable cardinalities. 
This was originally proved in \cite{Zilber} but an error was later found in the proof and corrected in \cite{bays}.
Throughout this article, we will make the assumption $K \cong \mathbb{Z}$.

We continue the work done in \cite{Lucy} and study the cover as a quasiminimal class. Introduction of these cover structures to model theory
was the starting point of the study of quasiminimal pregeometries,
and to understand these, it is important to understand the geometries in examples,
and cover structures provide a most natural starting point for this.
As our main result, we prove that smooth curves on the cover satisfy the axioms for a Zariski-like structure, which we presented in \cite{lisuriart} (see also \cite{lisuri}).  
However, to make notation simpler and the proofs more readable, we first give the result for the cover itself (which can be seen as a very simple curve) and then point out that the same arguments hold in the more general case.

Our study of Zariski-like structures was inspired by the concept of \emph{Zariski geometry}, introduced by E. Hrushovski and B. Zilber in \cite{HrZi} as a generalization of the Zariski topology of an algebraically closed field.
One of the results in \cite{HrZi} is that in a non locally modular, strongly minimal set in a Zariski geometry, an algebraically closed field can be interpreted. 
This result plays an important role in Hrushovski's proof of the geometric Mordell-Lang Conjecture (\cite{Hr}), where model-theoretic ideas were applied to solve a problem from arithmetic geometry.

In \cite{lisuriart} (see also \cite{lisuri}), we presented \emph{Zariski-like} quasiminimal structures as a generalization of Zariski geometries to a non-elementary context. The Zariski geometry axioms are given in terms of closed sets, but we decided to formulate the axioms for Zariski-like structures using properties of irreducible closed sets (which, for simplicity, we call just irreducible sets).
In the axiomatization, it is required that for a quasiminimal pregeometry structure (in the sense of \cite{monet}) $\M$, there is, for each $n$, a collection of subsets of $\M^n$, that we call the \emph{irreducible sets}, having certain qualities that generalize some properties of irreducible closed sets in a Zariski geometry.
There are nine axioms altogether, and they are listed in the beginning of section 5.
 
In \cite{HrZi}, it is pointed out that for finding a field, the full axiomatization for a Zariski geometry is not needed but it suffices that the structure satisfies certain weaker assumptions (Assumptions 6.6. in \cite{HrZi}).
Some of our axioms for a Zariski-like structure come from these assumptions rather than from the axiomatization (Z0)-(Z3) for Zariski geometries given in the very beginning of \cite{HrZi}.
In our axiomatization, axioms (1)-(6) are to give meaning to the key axioms  (7)-(9).
They list some very basic properties that irreducible closed sets are expected to satisfy in a context resembling Zariski geometry. The axioms (7) and (8) come from Assumptions 6.6. in \cite{HrZi} and state that specializations behave nicely enough. Axiom (9) brings back some traces of Compactness to the non-elementary context.

In \cite{lisuriart} (see also \cite{lisuri}), we showed that a group  can be interpreted in a Zariski-like structure given that the canonical pregeometry obtained from the bounded closure operator is non-trivial.
The axiomatization comes into play in the non locally modular case, as in non-trivial locally modular pregeometries, groups can be always found
as shown in \cite{lisuriart} (see also \cite{lisuri}). 
In a forthcoming paper, the second author will show that in the non
locally modular case an algebraically closed field can be interpreted
assuming that the structures do not interpret so called non-classical groups
(see \cite{HLS}). Of course,
in the case of the cover, it is clear that there always is a group
and a field present, but we recover them from the geometry alone, which is important in itself.
The properties of the groups that can be interpreted
in non-elementary cases depend on 
whether non-classical groups (see \cite{HLS}) exist, which is an open question.
This is discussed in more detail in \cite{HLS}.

The main reason for presenting the concept of Zariski-like
quasiminimal structures was to provide a context where we hope to be able to classify non-elementary geometries.
We introduced these structures as generalizations of Zariski geometries, and indeed, all Zariski geometries are Zariski-like (see Example 4.7 in \cite{lisuriart} or Example 4.7 in \cite{lisuri}).
It is also easy to construct artificial Zariski-like structures as in e.g. Example 4.8 in \cite{lisuriart} (Example 4.8 in \cite{lisuri}).
 Moreover, we would expect that structures similar to the cover presented in this article would be Zariski-like.
 Zilber's pseudo-exponentiation (see \cite{pseudo}), in particular, is an interesting candidate.
 In \cite{quantum}, Zilber discusses non-commutative quantum tori.
These structures arise from an algebra generated by the ``operators" $U$, $V$, $U^{-1}$ and $V^{-1}$, satisfying the relation $VU=qUV$ for some $q \in \mathbb{C}$, and are inspired by the quantum harmonic oscillator.
The structure is a Zariski geometry in case $q$ is a root of unity.
In the case that $q$ is of infinite order, it might be a Zariski-like structure.
   
Zilber himself suggest analytic Zariski structures as the right
level of generality in which the analysis can be made but that
appeared too difficult for us and thus our approach is more modest one.   
   
To be able to formulate the axioms of a Zariski-like structure for the cover, we need a concept of an irreducible set.
A suitable notion can be found utilizing the PQF-topology, originally presented in \cite{Lucy}.
It is obtained by taking as basic closed sets the sets definable by positive, quantifier-free formulae.
The desired properties are then expressed in terms of the irreducible PQF-closed sets definable over $\emptyset$ (after adding countably many symbols to the language).
The main result, i.e. that the cover is Zariski-like, is proved in section 5.
We then go on pointing out that all the arguments presented also apply for any curve $C$ on the cover as long as $\textrm{exp}(C)$ is a smooth algebraic curve.

In the arguments in section 5, we utilize our previous result from \cite{lisuriart}, namely that an independence calculus can be developed for abstract elementary classes (AECs) satisfying certain conditions.
Thus, we need to be able to construct for the theory of covers a model class that can be viewed as such an AEC.
Moreover, in applying the independence calculus, it is important that dimensions in the model theoretic sense can be obtained from the topological properties of the cover. 
In the earlier sections of this paper, we provide all the results enabling this. 
In section 2, we introduce the PQF-topology and present the properties of PQF-closed sets needed in the latter sections. These results have been previously given in \cite{Lucy}. 
 In section 3, we define a concept of dimension for the PQF-closed sets using their topological properties.
We then show that to be able to determine the dimension of a PQF-closed set $C$, it suffices to look at its image under the map exp.
In \cite{lisuriart}, we gave quasiminimal classes (in the sense of \cite{monet}) as an example of AECs satisfying the conditions required for our independence calculus.
Thus, in section 4, we view the cover as a quasiminimal pregeometry structure.
Moreover, we show that the pregeometry involved can be obtained from the bounded closure operator (bcl).
This plays an important role when applying the results from \cite{lisuriart} in section 5.
At the end of section 4, we also show that the dimension obtained topologically in section 3 agrees with the model-theoretic one obtained from bcl.
All the arguments  in this article can be found in more detail in Chapter 5 in \cite{lisuri}.
  
When using the word variety, we always mean a Zariski closed subset of $F^n$ for some $n$, defined as the zero locus of some set of polynomials.
That is, we only consider affine varieties, and we don't require them to necessarily be irreducible. 
Until the middle of section 4, the field can be an arbitrary algebraically closed field, but when we move to the monster, it is important that $F$ is large enough (at least uncountable).
From thereon, the reader may make the assumption $F=\mathbb{C}$.
Morley rank is denoted by $MR$.
We follow the usual convention in model theory and write (e.g.) $a \in V$ for ``$a \in V^n$ for some $n$".
 
We thank Jonathan Kirby for suggesting that a torus could be transformed into a canonical form by a change of coordinates.
  
\section{PQF-topology}

To be able to formulate the axioms of a Zariski-like structure for the cover, we need the concept of an irreducible set.
A natural notion is provided by the PQF-topology, introduced in \cite{Lucy}, section 3.2.1.
 
\begin{definition}
For each $n$, define a topology on $V^n$ by taking the sets definable by positive quantifier-free first-order formulae (using any parameters) as the basic closed sets.
Call this the \emph{PQF}-topology.
\end{definition}
 
We define the notion of an irreducible set in the usual way.

\begin{definition}
We say a nonempty closed set is \emph{irreducible} if it cannot be written as the union of two proper closed subsets.
\end{definition}

Above, we have allowed any parameters for defining closed sets.
However, when we formulate the axioms of a Zariski-like structure for the cover in section 5, we will take only the irreducible sets definable by a positive, quantifier-free formula with using only elements of $\textrm{log } \overline{\mathbb{Q}}$ (here $\overline{\mathbb{Q}}$ denotes the field of algebraic numbers) as parameters, to be the irreducible sets in the sense of the definition of a Zariski-like structure.
 
Since the PQF-topology is not Noetherian (indeed, $(C_i)_{i<\o}$, where $C_i=\left\{ u \in V \, | \, \textrm{exp}\left(\frac{u}{2^i} \right)=1 \right\}$ forms an infinite descending chain of PQF-closed sets), we cannot speak about irreducible components in the classical sense.
However, we give a more general definition of irreducible components that makes sense in the context of PQF-closed sets.
 
\begin{definition}
If $C$ is a PQF-closed set, we say that the \emph{irreducible components} of $C$ are the maximal irreducible subsets of $C$.
\end{definition}

Later in this section, we will see that this definition actually makes sense, i.e. that all PQF-closed sets have irreducible components (and actually only a countable number of them).
On the way, we will also see that all irreducible PQF-closed sets are actually basic closed sets, i.e. definable by positive, quantifier free formulae (note that from the definition of PQF-closed sets, it does not directly follow that all PQF-closed sets are definable, since infinite intersections of basic closed sets are closed). 
 
We will eventually prove that for any irreducible $PQF$-closed $C$ and any variety $W$ on the field, it holds that $C$ is an irreducible component of $\textrm{log}(W)$ if and only if $\textrm{exp}(C)$ is an irreducible component of $W$ (in the classical sense):

\begin{corollary} 
If $W\subseteq \mathbb{C}^n$ is a variety and $C\subseteq V^n$ is an irreducible subset of the cover, then $C$ is an irreducible component of $\textrm{log }W$ if and only if $\textrm{exp}(C)$ is an irreducible component of $W$.
\end{corollary}

This result is given in Corollary \ref{olennainen} at the end of this section.
It will be used when proving the results on dimension theory that are needed for our main result.
We will obtain it as a corollary for a theorem that states that for an irreducible variety $W$, the irreducible components of $\textrm{log}(W)$ can be presented in a canonical form (Theorem \ref{komponentit}).
The theorem in itself is also important when analyzing the PQF-topology on the cover.
 
We will now give some results about the cover as a topological structure that will be needed in the proof (or in later sections of this paper).
Everything has been previously presented in \cite{Lucy}, chapter 3, and we refer the reader to \cite{lisuri}, chapter 5, for omitted details.
  
When looking at the first-order type of some tuple $v$ on the cover, we usually start by looking at the smallest variety containing $\textrm{exp}(v)$.
But in addition to this, we need to know the smallest variety containing $\textrm{exp}(\frac{v}{n})$ for each natural number $n$. 
This leads us to consider roots of varieties.
  
\begin{definition}
Let $W$ be an irreducible variety.
For any natural number $n$, we say that an irreducible variety $X$ is an $n$:th root of $W$ if 
$$X^n=\{x^n \, | \, x \in X\}=W.$$ 
 
Suppose now $W$ is an arbitrary variety with a decomposition $W=W_1 \cup \ldots \cup W_r$ into irreducible components.
Then, we define the $n$:th roots of $W$ to be all the unions of the form $\bigcup_{i=1}^r {{W_i}^{\frac{1}{n}}}_{(j)}$, where each ${{W_i}^{\frac{1}{n}}}_{(j)}$ is an $n$:th root of $W_i$.
\end{definition}
 
We note that all varieties have $n$:th roots for every $n$.
Indeed, suppose $W$ is irreducible, and let $Y_1, \ldots, Y_n$ be the irreducible components of the set $Y=\{x \, | \, x^n \in W\}$. 
We may write
$$W=X^n=Y_1^n \cup \ldots \cup Y_r^n.$$
As $W$ is irreducible, we have $Y_i^n=W$ for at least one $i$.

We also note that every $n$:th root of $W$ must be an irreducible component of $Y$, and thus, in particular, every variety has only finitely many $n$:th roots.
Indeed, let $Y'$ be some $n$:th root of $W$.
Then, $Y'$ is irreducible, and thus $Y' \subseteq Y_i$ for some $i$.
If $Y' \subsetneq Y_i$, then $dim(Y')<dim(Y)=dim(W)$ ($dim$ denotes the algebraic dimension) which contradicts the assumption that $Y'$ is a $n$:th root of $W$ as finite-to-one maps preserve the dimension.
   .
Also, if $W$ is any variety and $X$ is a $n$:th root of $W$, then clearly $X^n=W$. 

Let $W\subset F^n$ be an irreducible variety with $m$:th roots $W^{\frac{1}{m}}_{(i)}$. 
We say that an element $x \in F^n$ is an $m$:th \emph{root of unity} if each of its coordinates is an $m$:th root of unity in $F$.
It is easy to see that multiplication by $m$:th roots of unity permutes the $m$:th roots of $W$.
  
One easily sees that the following lemma holds. 
 
\begin{lemma}\label{tyypit}[\cite{ZilUusi}]  
Let $V$ be a cover and $A \subset V$.
Let $v \in V^n$ with linearly independent coordinates.
Then, the quantifier free type of $v$ over $A$ is determined by the formulae
\begin{eqnarray*}
\textrm{exp}\left(\frac{v}{l}\right) &\in& W^{\frac{1}{l}} \quad l \in \mathbb{N}, \\
\textrm{exp}(v) &\notin& Y \quad Y \subset W, \textrm{dim}(Y)<\textrm{dim}(W), \\
mv &\neq& 0 \quad m \in \mathbb{Z}^n, m \neq 0,
\end{eqnarray*}
where $W$ is the locus of $\textrm{exp}(v)$ over $\mathbb{Q}(\textrm{exp}(A))$ (the smallest field containing $\textrm{exp}(A))$ and each $W^{\frac{1}{l}}$ is an $l$:th root of $W$.
\end{lemma}

To determine quantifier-free types on the cover, we need to be able to take account of all the possible roots that a given variety has.
Thus, we introduce the concept of branching.

\begin{definition}
Let $W$ be a variety.
If $W$ has distinct $n$:th roots for some natural number $n$, we say that $W$ \emph{branches}.

We say that $W$ \emph{stops branching at the finite level} if there is a natural number $l$ such that the $l$:th roots $W^{\frac{1}{l}}$ no longer branch.

We say $W$ \emph{branches infinitely} if it does not stop branching at the finite level.
\end{definition}

When studying branching more closely, we will see that it makes a great difference whether a variety is contained in a torus or not.
A torus can be seem as a coset of a subgroup of $({F^*})^n$ equipped with a group operation given by coordinate-wise multiplication.
This gives the following definition.

\begin{definition}
 Call a set $T \subseteq (F^*)^m$ a \emph{torus} if it can be defined using equations of the form
\begin{displaymath}
\prod_i {x_i}^{z_i}=c \qquad z_i \in \mathbb{Z} \quad c \in F^{*} 
\end{displaymath}

If $T \subseteq (F^*)^m$ is a torus such that $T \neq (F^*)^m$, we say that $T$ is a \emph{proper torus}.
\end{definition}

We will sometimes view a torus as a variety.
Then, we mean the Zariski closure of a set that is defined as above.
We will say that a torus $T$ is \emph{irreducible}, if the Zariski closure of $T$ is irreducible as a variety (in the usual sense).
 
Jonathan Kirby pointed out to us that irreducible tori can be transformed in a canonical form by a birational coordinate change on $(F^{*})^n$, as stated in the following lemma.
 
\begin{lemma}\label{torusbirat}
Let $T \subset( {F^{*}})^m$ be an irreducible torus given in the coordinates $x_1, \ldots, x_m$.
Then, there is a birational coordinate change given by 
\begin{eqnarray}\label{chaange}
y_i=\prod_{j=1}^m x_j^{z_{ij}}, \quad x_i=\prod_{j=1}^m y_j^{z_{ij}'}, \quad 1 \le i \le m, \quad z_{ij}, z_{ij}' \in \mathbb{Z},
\end{eqnarray}
such that in the new coordinates, $T$ is of the form 
$$y_i=c_i, \quad 1 \le i \le k,$$
where $1 \le k \le m$, and $c_i \in F^*$ for each $i$.
\end{lemma}
 
\begin{proof} 
Suppose the torus $T$ is given by the equations 
\begin{eqnarray}\label{tallaa}
\prod_{i=1}^m x_i^{n_{ji}}=c_j, 
\end{eqnarray}
where $1\le j \le k$ for some $k$, $n_{ij} \in \mathbb{Z}$ and $c_j \in F^*$.
 
To prove the lemma, we will view the multiplicative group $(F^{*})^m$ as a $\mathbb{Z}$-module where $\mathbb{Z}$ acts by exponentiation. 
Then we  look for invertible endomorphims that would give a suitable coordinate change. 
 
We start looking at the first one of the equations (\ref{tallaa}). 
Since $T$ is irreducible, $1$ is the 
greatest integer that divides each one of the $n_{1i}$ for $1 \le i \le m$.
%For each $i$, denote $n_{1i}'=\frac{n_{1i}}{d_1}$.
Thus, there exists an integer matrix $A$ such that the first row of $A$ is $(n_{11}, \ldots, n_{1m})$ and $\textrm{det}(A)=1$ (for details, see Lemma 5.8 in \cite{lisuri}).
Then, the coordinate change given by $A$ is of the form (\ref{chaange}) and transforms our equation into $y_1=c_1$.
Using Cramer's rule, we see that  $A^{-1}$ is also an integer matrix, 
%(remeber that $A^{-1}=\frac{1}{\textrm{det}(A)} C^T$, where $C_{ij}=(-1)^{i+j} M_{i,j}$ for the minor $M_{i,j}$ of $A$).
and thus the reverse coordinate change is also given in the form (\ref{chaange}).
 
Since the coordinate change we have done is given by equations of the form (\ref{chaange}), 
all the equations in (\ref{tallaa}) are still in the torus form after the transformation.
Consider the second equation.
After substituting $y_1=c_1$, it will be given by 
\begin{eqnarray*} 
y_2^{z_2} \cdots y_m^{z_m}=c_2'
\end{eqnarray*}
for some $z_2, \ldots, z_m \in \mathbb{Z}$, $c_2' \in F^*$. 
Let $d$ be the greatest integer dividing each one of the numbers $z_2, \ldots, z_m$.
Then, an integer matrix with determinant $ \pm 1$ and first row $(\frac{z_2}{d}, \ldots, \frac{z_m}{d})$ transforms the equation into
$$y_2^{d}=c_2',$$
which gives us 
$$y_2=\zeta^{i} a, \quad i=0, \ldots, d-1,$$
where $a$ is a number such that $a^{d}=c_2'$ and $\zeta$ is a primitive $d$:th root of unity.

We substitute these values to the third equation to get at most $d$ distinct equations.
Then, we deal with each one of them as we did with the second equation above.
Proceeding this way and going through all the equations, we will get $T$ in the new coordinates as a union of smaller tori, each given by equations of the form
$$y_i=c_i, \quad 1 \le i \le k$$
for $c_i \in F^*$ and some $k \le n$.
Since our coordinate change and its inverse are both given by rational functions, it is a homeomorphism in the Zariski topology, and thus maps irreducible sets to irreducible sets.
Since we assumed $T$ to be irreducible, only one of the components listed is nonempty. 
This proves the lemma.
\end{proof}

Now it is easy to prove the following.  

\begin{lemma}\label{torusom}
The following hold:
\begin{enumerate}[(a)]
\item If $T_1$, $T_2$ are tori, then $T_1 \cap T_2$ is a torus.
\item If $T$ is a torus, then every irreducible component of $T$ is a torus.
\item If $T$ is a torus, then $T$ has distinct $m$:th roots for any $m$.
Moreover, any $m$:th root of $T$ is a torus.
\item If $T$ is an irreducible torus, then $T^m$ is a torus for every natural number $m$.
\end{enumerate}
\end{lemma}
  
It is a consequence of the following theorem that varieties do not branch infinitely as long as they are not contained in a torus.  

 \begin{theorem}\label{ZilberTheorem}  
Let $V$ be a cover with ($K=\mathbb{Z}$).
Let $G$ be a countable submodel such that $G=\textrm{log}(\textrm{exp}(G))$ and $\textrm{exp}(G)$ is an algebraically closed field.
Let $h \in V^m$.
Let $W$ be the locus of $\textrm{exp}(h)$ over $\textrm{exp}(G)$.
Suppose $W$ is not contained in a torus. 
Then the subtype of $h$ over $G$ consisting of formulae $\textrm{exp}(\frac{h}{l}) \in W^\frac{1}{l}$ is determined by a finite number of formulae.
 \end{theorem}

\begin{proof}
We may without loss assume $h \notin G$ (otherwise it is trivial to prove the statement).
By Theorem 3 in \cite{bays}, there exists some $m \in \mathbb{N}$ such that if $(c_1^{\frac{1}{n}})_n$ and $(c_2^{\frac{1}{n}})_n$ are division systems below $\textrm{exp}(h)$ (see \cite{bays}, Definition 2) with $c_1^{\frac{1}{m}}=c_2^{\frac{1}{m}}$, then, for all $l \in \mathbb{N}$, we have that $c_1^{\frac{1}{l}}$ and $c_2^{\frac{1}{l}}$ have the same field type over $\textrm{exp}(G)$.
Since $W$ has only finitely many $m$:th roots, the formula $\textrm{exp}(\frac{h}{m}) \in W^{\frac{1}{m}}$ determines the whole subtype. 
 \end{proof}

Suppose $W$ is a variety, $v \in \textrm{log }W$.
 For any $l$, denote by $W^{\frac{1}{l}}_{(v)}$ the $l$:th root of $W$ such that $\textrm{exp}(\frac{v}{l}) \in W^{\frac{1}{l}}_{(v)}$.
If $W$ is a variety not contained in any torus, then, Theorem \ref{ZilberTheorem} implies that there is some number $m$ such that for any $m'>m$, the $m'$:th root $W^{\frac{1}{m'}}_{(v)}$ is determined by the $m$:th root $W^{\frac{1}{m}}_{(v)}$.

On the other hand, if $W$ is contained in a torus $T$, taking account of roots becomes more complicated since tori branch infinitely.
But since we are looking at types on the cover $V$, we can solve this problem by finding a nice enough set $L \subseteq V$ 
such that $\textrm{exp}(L)=T$. 
 
\begin{definition}
A set $L \subseteq V^n$ is called \emph{linear} if it is a definable affine subspace of $V^n$. 
\end{definition}

\begin{remark}\label{linearintersect}  
For any linear set $L \subset V^n$ and any $k=(k_1, \ldots, k_n) \in K^n$, it holds that $L \cap (L+k) \neq \emptyset$ if and only if $L+k=L$.
\end{remark}
 
\begin{remark}  
We note that if $L \subset V^n$ is linear, then $\textrm{exp}(L)\subset F^n$ is a torus.
Also, using Lemma \ref{torusbirat}, it is easy to see that 
any irreducible torus $T\subset F^n$ can be written as $T=\textrm{exp}(L)$ for some linear set $L\subset V^n$. 
\end{remark}

\subsection{Irreducible sets}

We will now start looking more closely at the properties of PQF-closed sets that will be needed when analyzing irreducible components.   
 
\begin{lemma}[\cite{Lucy}, Lemma 3.2.1]\label{PQFbasicform}
Any set definable by a positive quantifier free formula is a finite union of sets of the form
\begin{displaymath}
m \cdot (L \cap \textrm{log } W)
\end{displaymath}
for some linear set $L$, a variety $W$ and some $m \in \mathbb{N}$.
\end{lemma}
 
\begin{corollary}\label{variety}[\cite{Lucy}, Corollary 3.2.1]
Let $C$ be a set on the cover, definable by a positive, quantifier-free formula (i.e. a basic closed set in the PQF-topology).
Then $\textrm{exp}(C)$ is a Zariski closed set on the field.
\end{corollary}

We will show that all irreducible sets are actually definable by positive quantifier-free formulae.
First, we give a canonical way to write any irreducible  variety $W$ as $W=T \cap W'$, where $T$ is a torus and $W'$ is not contained in any proper torus.
This will allow us to deal with roots nicely - using a suitable linear set $L$ when dealing with $T$ and 
Theorem \ref{ZilberTheorem} when dealing with $W'$.

\begin{lemma}\label{leikkaus}
Any irreducible variety $W \subset F^n$ can be written as $W' \cap T$ where $T$ is the minimal torus containing $W$ (note that this could be $F^n$) and $W'$ is an irreducible  variety not contained in any proper torus.
\end{lemma}

\begin{proof}  
By Lemma \ref{torusbirat}, we may assume $T$ is (without loss of generality) given by equations $x_i=c_i$ for $1\le i \le k$ where $k \le n$.
Let $a$ be a generic point of $W$ and let $I \subseteq F[x_{k+1}, \ldots, x_{k_n}]$ be the ideal consisting of all polynomials $f$ such that $f(a)=0$.
Let $J=\langle I \rangle \subseteq F[x_1, \ldots, x_n]$, the ideal generated by $I$ in $F[x_1, \ldots, x_n]$.
Let $W'$ be the variety associated to $J$.
Since the ideal $J$ does not contain any of the polynomials $x_i-c_i$ for $1 \le i \le k$ and since $T$ is the minimal torus containing $W$, the variety $W'$ is not contained in any torus.

It remains to show that $W'$ is irreducible. Indeed, since $W$ is irreducible, also the variety $V(I)$ given by the ideal $I$ is irreducible.
Now, $W'=F^k \times V(I)$ which is irreducible. 
\end{proof}
  
We still need two lemmas before being able to show that irreducible sets are definable by positive, quantifier-free formulae.
 
\begin{lemma}\label{singlemth}
 Let $T$ be an irreducible torus and let $L\subset V^n$ be a linear set such that $\textrm{exp}(L)=T$.
Then, for each $m$, $\textrm{exp}(\frac{L}{m})=X$, where $X$ is an $m$:th root of $T$.
\end{lemma}

\begin{proof}
By Lemma \ref{torusbirat}, we may assume $T$ is given by a finite set of equations of the form
\begin{eqnarray*}  
x_i-c_i=0,
\end{eqnarray*}
$1 \le i \le k$, where $1 \le k \le n$ and $c_i \in F^*$ for each $i$.
Let $\zeta_{ij}$ ($j=1, \ldots, m$) be the $m$:th roots of $c_i$.
An $m$:th root of $T$ then satisfies, for each $i$, exactly one of the equations
\begin{eqnarray*} 
x_i-\zeta_{ij}=0.
\end{eqnarray*}
 If $\frac{L}{m}$ were to contain some elements $a=(a_1, \ldots, a_n)$ and $b=(b_1, \ldots, b_n)$ that would map into distinct roots under exp, then  
$\textrm{exp}(a_i) \neq \textrm{exp}(b_i)$ for some $1 \le i \le k$.
This is impossible since both $a$ and $b$ satisfy the set of linear equations giving $\frac{L}{m}$.
\end{proof}
 
\begin{lemma}\label{linearirre}
Let $C \subset V^n$ be PQF-closed and irreducible, and let $L \subset V^n$ be linear.
Suppose $C \subset \bigcup_{k \in K^n} L+k$.
Then, there exists some $k \in K^n$ such that $C \subset L+k$.
 \end{lemma}

\begin{proof}
Translating $L$ and $C$ if needed, we assume towards a contradiction that $L$ is a $\mathbb{Q}$- vector space and there exist $a \in C \cap L$ and $b \in C \cap (L+k)$ for some $k \in K^n \setminus L$.
Denote $T=\textrm{exp}(L)$.
Irreducibility of $C$ implies that $a$ and $b$ map under exp to the same irreducible component of $T$, say, $T'$.   
Let $L' \subseteq L$ be a linear set such that $\textrm{exp}(L')=T'$ and $a \in L'$.
We may replace $L$ with $L'$.
By Lemma \ref{singlemth}, $\frac{L'}{M}$ is mapped onto a single $M$:th root of $T'$ under the map exp for each $M$, and so is $\frac{L'+k}{M}$.
Choosing $M$ large enough, these roots will be distinct.
Thus, if $T_1^{\frac{1}{M}}, \ldots, T_r^{\frac{1}{M}}$ are the $M$:th roots of $T$, then $\bigcup_{i=1}^r (\textrm{log}(T_i^{\frac{1}{M}} \cap C)$ is a non-trivial decomposition of $C$.
\end{proof}

The following lemma gives a canonical form for the irreducible sets.
It also implies that in particular, they are definable.
 
\begin{lemma}\label{irredform}[\cite{Lucy}, Lemma 3.2.3]
An irreducible PQF-closed set $C \subseteq V^n$ has the form 
\begin{displaymath}
C=L \cap m \cdot \textrm{log } W,
\end{displaymath}
where $L$ is a linear set, $W$ is a variety that does not branch, and $m \in \mathbb{N}$. 
\end{lemma}

\begin{proof}
Using Lemma \ref{PQFbasicform}, we may without loss assume
\begin{displaymath}
C=\bigcap_{i<\kappa} (L_i \cap m_i \textrm{log } W_i)
\end{displaymath}
for some cardinal $\kappa$.
By Noetherianity of the linear topology on $V^n$, the linear part stabilizes, so writing $L=\bigcap_{i < \kappa} L_i$, we get
\begin{eqnarray}\label{cee}
C= L \cap \bigcap_{i <\kappa} m_i \textrm{log } W_i.
\end{eqnarray}
Using Lemmas \ref{leikkaus} and \ref{linearirre}, we may further assume that each $W_i$ is an irreducible variety not contained in any torus and that it does not branch. 

Consider the intersections $W_I=\bigcap_{i \in I} \textrm{log } W_i^{\frac{m_i}{M_I}}$, where $M_I=\prod_{i \in I} m_i$ and $I$ ranges over the finite subsets of $\kappa$.
We show that we may choose (after modifying the intersects) $I$ so that $W_I$ is irreducible with minimal Morley rank.
Since $C$ is irreducible, we may without loss assume that $W_I$ is irreducible (if needed, replace it with the irreducible component containing $\textrm{exp}(\frac{C}{M_I})$).
If we wish to enlarge $I$ into $I \cup \{i\}$ for some $i<\kappa$, we first take $W_I^{\frac{1}{m_i}}$ (note that $W_I$ might branch, so this is a union of roots) and replace it by $W_I'$, the $m_i$:th root of $W_I$  
containing $\textrm{exp}(\frac{C}{m_iM_I})$. 
We have $MR(W_I')=MR(W_I)$, and since $W_I'$ is irreducible, $MR(W_I' \cap W_i^{\frac{1}{M_I}})<MR(W_I')$.
Proceeding this way and discarding redundant elements from (\ref{cee}), we will eventually find some finite $I$ such that the variety $W_I$ obtained in this process has minimal Morley rank and is irreducible.

Now, we have $\textrm{exp}(\frac{C}{M_I})=T \cap W_I$, where $T=\textrm{exp}(\frac{L}{M_I})$.
%Since $\frac{C}{M_I}$ is irreducible, so is $\textrm{exp}({\frac{C}{M_I}})$.
By Lemma \ref{leikkaus}, we may write $T \cap W_I =T' \cap W'$ for some torus $T' \subseteq T$ and some variety $W'$ not contained in any torus. 
Let $N$ be a number such that the $N$:th roots of $W'$ no longer branch, and let $L' \subseteq \frac{L}{M_I}$ be a linear set such that $\textrm{exp}(L')=T'$.
Then, $\frac{C}{M_I N}=\frac{L'}{N} \cap \textrm{log }W$, where $W$ is some $N$:th root of $W'$, so $C$ is of the desired form.
%Jos niitŠ muita olisi epŠtriviaalisti otettavissa mukaan, ne pudottaisivat Morley Rankia (kts lisuri, sivu 94, kohdsta "consider the representation (5.5).
\end{proof}

\begin{corollary}
If $C$ is a closed irreducible subset of the cover, then $\textrm{exp}(C)$ is an irreducible variety.
\end{corollary}

We are aiming prove that for any irreducible $PQF$-closed $C$ and any variety $W$, it holds that $C$ is an irreducible component of $\textrm{log}(W)$ if and only if $\textrm{exp}(C)$ is an irreducible component of $W$.
We first show that for an irreducible variety $W$, the irreducible components of $\textrm{log}(W)$ are of a certain form (Theorem \ref{komponentit}).
When we study a variety $W$, we will from now on write it as $W=T \cap W'$ where $T$ is the minimal torus containing $W$ and $W'$ is a variety not contained in a torus.
Moreover, we will always assume the variety $W'$ is obtained as in the proof of Lemma \ref{leikkaus}.

We will need the following auxiliary result in the proof of Theorem \ref{komponentit}, which will give us more information of the irreducible components of $\textrm{log } W$ for any variety $W$.
 
 \begin{lemma}\label{rootsintersect}
Let $W=T\cap W' \subset F^n$ be an irreducible variety, and let $m$ be a natural number.
Then, the $m$:th roots of $W$ are exactly the varieties $T^{\frac{1}{m}}_{(i)} \cap W'^{\frac{1}{m}}_{(j)}$, where
 $T^{\frac{1}{m}}_{(i)}$ goes through the $m$:th roots of $T$ and $W'^{\frac{1}{m}}_{(j)}$ goes through the distinct $m$:th roots of $W'$.
\end{lemma}

\begin{proof} 
Let $X$ be a $m$:th root of $W$.
Write $X=T' \cap Y$, where $T'$ is the minimal torus containing $X$ and $Y$ is a variety not contained in any torus, obtained as in the proof of Lemma \ref{leikkaus}.
There is some $m$:th root $T^{\frac{1}{m}}_{(i)}$ of $T$ such that $X \cap T^{\frac{1}{m}}_{(i)} \neq \emptyset$.
Then,
$(T^{\frac{1}{m}}_{(i)} \cap X)^m =W.$
Hence, 
$\textrm{MR}(T^{\frac{1}{m}}_{(i)} \cap X)=\textrm{MR}(W)=\textrm{MR}(X),$  
so $X \cap T^{\frac{1}{m}}_{(i)}=X$ by irreducibility of $X$.
Thus $T' \subseteq T^{\frac{1}{m}}_{(i)}$ by the minimality of $T'$.
On the other hand, $W \subseteq T'^m$ and hence $T \subseteq T'^m$.
So, $T' = T^{\frac{1}{m}}_{(i)}$.

To prove that $Y$ is an $m$:th root of $W'$, it suffices, by irreducibility of $Y$, to show that $Y^m=W'$.
By Lemma \ref{torusbirat}, we may assume that $T$ is give by equations $x_i=c_i$, $1 \le i \le k$ for some $k \le n$, where $c_i \in F^{*}$, and that $W'$ is given by the ideal $J_{W'}=\langle I_{W'} \rangle$, where $I_{W'}$ consists of all the polynomials $f \in F[x_{k+1}, \ldots, x_n]$ such that $f(a)=0$ for a generic point $a \in W$.
Let $J_Y=\langle I_Y \rangle$ be the ideal corresponding to $Y$, obtained similarly.
For every $f \in I_{W'}$, the polynomial $f(x_{k+1}^m, \ldots, x_n^m)$ is in $I_Y$, so $Y^m \subseteq W'$.
To prove $W' \subseteq Y^m$, we show that $I(Y^m) \subseteq J_{W'}$ for $I(Y^m)$ the ideal corresponding to $Y^m$.
Let $g$ be a generator of $I(Y^m)$.
Then, $g \in F[x_{k+1}, \ldots, x_n]$ and $g(x_{k+1}^m, \ldots, x_n^m) \in I(Y) \subseteq I(X)$.
Since $W=X^m$, we have by construction of $W'$ that $g \in J_{W'}$.
Thus, $Y^m=W'$ and $Y$ is an $m$:th root of $W'$
This concludes the first part of the proof.
  
Let now $T^{\frac{1}{m}}_{(i)}$ be an $m$:th root of $T$ and $W'^{\frac{1}{m}}_{(j)}$ an $m$:th root of $W'$.
We show that $T^{\frac{1}{m}}_{(i)} \cap W'^{\frac{1}{m}}_{(j)}$ is an $m$:th root of $W$.
Since
$(T^{\frac{1}{m}}_{(i)} \cap W'^{\frac{1}{m}}_{(j)})^m=W,$
there is some $m$:th root $X$ of $W$ such that  $X \subseteq T^{\frac{1}{m}}_{(i)} \cap W'^{\frac{1}{m}}_{(j)}.$
We have already seen that $X=T^{\frac{1}{m}}_{(i')} \cap W'^{\frac{1}{m}}_{(j')}$ for some $m$:th roots $T^{\frac{1}{m}}_{(i')}$ and $W'^{\frac{1}{m}}_{(j')}$ of $T$ and $W'$, respectively.
Clearly, $T^{\frac{1}{m}}_{(i')} =T^{\frac{1}{m}}_{(i)}$.

Suppose for the sake of contradiction that $X \subsetneq T^{\frac{1}{m}}_{(i)} \cap W'^{\frac{1}{m}}_{(j)}.$ 
Let $\zeta=(1, \ldots, 1, \zeta_{k+1}, \ldots, \zeta_n)$ be an $m$:th root of unity such that $W'^{\frac{1}{m}}_{(j)}=\zeta W'^{\frac{1}{m}}_{(j')}$.
Since $\zeta T^{\frac{1}{m}}_{(i)}=T^{\frac{1}{m}}_{(i)}$, we get, multiplying
the equation
$$T^{\frac{1}{m}}_{(i)} \cap W'^{\frac{1}{m}}_{(j')} \subsetneq  T^{\frac{1}{m}}_{(i)} \cap \zeta W'^{\frac{1}{m}}_{(j')}$$
successively by powers of $\zeta$,  
\begin{eqnarray*}
T^{\frac{1}{m}}_{(i)} \cap W'^{\frac{1}{m}}_{(j')} \subsetneq  T^{\frac{1}{m}}_{(i)} \cap \zeta W'^{\frac{1}{m}}_{(j')}
\subsetneq T^{\frac{1}{m}}_{(i)} \cap \zeta^2 W'^{\frac{1}{m}}_{(j')} \subsetneq
\ldots \subsetneq T^{\frac{1}{m}}_{(i)} \cap \zeta^m W'^{\frac{1}{m}}_{(j')}= T^{\frac{1}{m}}_{(i)} \cap W'^{\frac{1}{m}}_{(j')},
\end{eqnarray*}
a contradiction.
\end{proof}

\begin{theorem}\label{komponentit}
Let $W$ be an irreducible variety on the field, and let $W=W' \cap T$ where $T$ is the minimal torus containing $W$, and $W'$ is contained in no torus.
Let $m$ be the level at which $W'$ stops branching and let $W'^{\frac{1}{m}}_{(i)}$ be the $m$:th roots of $W'$.
Let $L$ be linear such that $\textrm{exp}(L)=T$.
Then, the irreducible components of $\textrm{log } W \subset V^n$ are
\begin{displaymath}
(L+k) \cap m \cdot (\textrm{log } W'^{\frac{1}{m}}_{(i)}) \qquad k \in K^n.
\end{displaymath}
\end{theorem}

\begin{proof}  
We first prove that the components listed in the statement are irreducible. 
Suppose $C=L \cap m\textrm{log }W$, where $L$ is linear, $W$ does not branch, and $\textrm{exp}(C)=T \cap W^m$ is irreducible with $T=\textrm{exp}(L)$ the minimal torus containing $\textrm{exp}(C)$.
We wish to reduce to the case $C=L \cap \textrm{log }W$ by replacing $C$ with $\frac{C}{m}$.
To do this, we need to make sure that $\textrm{exp}(\frac{C}{m})$ is irreducible.
But $\textrm{exp}(\frac{C}{m})=T^{\frac{1}{m}} \cap W$, where $T^{\frac{1}{m}}$ is a single $m$:th root of $T$.
By Lemma \ref{rootsintersect}, this is an $m$:th root $T \cap W^m$ and thus irreducible. 
 
So we may suppose  $C=L \cap \textrm{log }W$.
Assume $C \subseteq \bigcup_{i=1}^r L_i \cap m_i \textrm{log }W_i$ for linear $L_i$ and varieties $W_i$.
We will show that one of the components of the union contains $C$.
Let $M=\prod_{i=1}^r m_i$ and $M_i=\frac{M}{m_i}$.
Since $W$ does not branch, $W^{\frac{1}{M}}$ is irreducible, so we may replace $C$ with $\frac{C}{M}$ as we did above. 
Moreover, $\textrm{exp}(\frac{1}{M_i} \textrm{log }W_i) =W_i^{\frac{1}{M_i}}$ (here, $W_i^{\frac{1}{M_i}}$ stands for the union of all roots), so we can reduce to the case where $L \cap \textrm{log }W \subseteq \bigcup_{i=1}^r L_i \cap \textrm{log }W_i$.
Intersecting the right hand side with $L$, we may assume that $L_i \subseteq L$ for each $i$.
Intersecting with $\textrm{log }W$ we may assume that $\textrm{log }W_i \subseteq \textrm{log }W$.
Thus, $\textrm{exp}(L_i \cap \textrm{log }W_i) \subseteq T \cap W$ for all $i$, and we have
$$T \cap W \subseteq \bigcup_{i=1}^n \textrm{exp}(L_i \cap \textrm{log }W_i) \subseteq T \cap W,$$
so $T \cap W =\bigcup_{i=1}^n \textrm{exp}(L_i \cap \textrm{log }W_i)$. 
Since it is irreducible, we must have $T \cap W=\textrm{exp}(L_i \cap \textrm{log }W_i)$ for some $i$.
From minimality of $T$, it follows that $L=L_i$.
Then, $W \cap T = W_i \cap T$, and we are done.
  
For maximality, let $L \cap m\textrm{log} W'^{\frac{1}{m}}_{(i)} \subsetneq I \subseteq \textrm{log } W$, where $I$ is a PQF-closed set.
We show that $I$ is reducible.
Either there is some $k \in K^n$ such that $L+k \neq L$ and $(L+k) \cap I \neq \emptyset$, and reducibility follows from Lemma \ref{linearirre}, or $I \cap (m\textrm{log }W'^{\frac{1}{m}}_{(j)} \setminus m \textrm{log }W'^{\frac{1}{m}}_{(i)}) \neq \emptyset$ for some $j \neq i$, in which case we can reduce $I=\bigcup (I \cap m \textrm{log }W'^{\frac{1}{m}}_{(i)})$.
\end{proof}
 
\begin{corollary}\label{olennainen}
If $W\subseteq F^n$ is a variety and $C\subseteq V^n$ is an irreducible subset of the cover, then $C$ is an irreducible component of $\textrm{log }W$ if and only if $\textrm{exp}(C)$ is an irreducible component of $W$.
\end{corollary}

\begin{proof}
It suffices to prove the claim when $W$ is irreducible (if it is not, reduce it).
But then this is immediate from Theorem \ref{komponentit} and Lemma \ref{irredform}.
\end{proof}

\section{Dimension Analysis}

In this section, we define dimensions for the PQF-closed sets much in the same way that they are defined in the Noetherian case.
We then prove that the dimension of a PQF-closed set is equal to the dimension of its image under the map exp.
 
\begin{lemma}\label{infchains}
There are no infinite descending chains of irreducible closed subsets of the cover.  
\end{lemma}

\begin{proof}
If $C_1$, $C_2$ are closed and irreducible, then by Corollary \ref{olennainen}, $C_1 \subsetneq C_2$ implies $\textrm{exp}(C_1) \subsetneq \textrm{exp}(C_2)$ (indeed, if $\textrm{exp}(C_1)=\textrm{exp}(C_2)=W$, then by Corollary \ref{olennainen}, both $C_1$ and $C_2$ are irreducible components of $\textrm{log }W$, so we cannot have $C_1 \subsetneq C_2$).
The claim then follows from Notherianity of the Zariski topology. 
\end{proof}
 
 We now define the dimension for irreducible closed sets exactly the same way that it can be defined in the Zariski topology.
 
\begin{definition}
If $C$ is an irreducible, closed and nonempty set on the cover, we define the \emph{dimension} of $C$ inductively as follows: 
\begin{itemize}
\item $\textrm{dim}(C) \ge 0$,
\item $\textrm{dim}(C)=\textrm{sup }\{\textrm{dim}(F)+1 \, | \, F \subsetneq C, \textrm{$F$ closed, irreducible and nonempty } \}$.
\end{itemize}
\end{definition}
 
\begin{definition}
Let $a \in V^n$, $A \subset V$.
By the \emph{locus} of $a$ over $A$, we mean the smallest PQF-closed subset definable over $A$ containing $a$.
When not specified, we assume the set $A$ to be empty. 

We define $\textrm{rk}(a/A)=\textrm{dim}(C)$, where $C$ is the locus of $a$ over $A$.
We write $\textrm{rk}(a)$ for $\textrm{rk}(a/\emptyset)$.
\end{definition}

We now prove that the dimension of an irreducible PQF-closed set equals the dimension of its image under the map exp (where the latter is calculated similarly as the former, using the Zariski closed sets on the field).

\begin{lemma}\label{dimensio}
Let $C \subset V^n$ be an irreducible PQF-closed subset of the cover.
Then,
\begin{displaymath}
\textrm{dim}(C)=\textrm{dim}(\textrm{exp}(C)).
\end{displaymath} 
\end{lemma}

\begin{proof} 
It is easy to see that $\textrm{dim}(C) \le \textrm{dim}(\textrm{exp}(C))$.
The other inequality is proved by induction on $\textrm{dim}(\textrm{exp}(C))$.
If  $\textrm{dim}(\textrm{exp}(C))=0$, then, using Corollary \ref{olennainen}, one sees that 
 $\textrm{dim}(C)=0$.
Suppose now $\textrm{dim}(\textrm{exp}(C)) \le \textrm{dim}(C)$ whenever 
$\textrm{dim}(\textrm{exp}(C))=n$.
Assume $\textrm{dim}(\textrm{exp}(C))=n+1$.
Denote $\textrm{exp}(C)=W$.
There is some irreducible $W' \subsetneq W$ such that $\textrm{dim}(W')=n$, and thus some 
irreducible $Y \subsetneq C$ such that $\textrm{exp}(Y)=W'$.
By the inductive hypothesis, $\textrm{dim}(Y) =n.$
Hence, $\textrm{dim}(C) \ge n+1$.
\end{proof}

In a Noetherian topology, closed sets have finitely many irreducible components.
In our context, we have the following analogy.

\begin{lemma}
Any basic PQF-closed set has countably many irreducible components.
\end{lemma}

\begin{proof}
Let $D$ be basic a PQF-closed set, and let $W=\textrm{exp}(D)$.
For each irreducible component $W'$ of $W$, let $C_i$, $i<\omega$, be the irreducible components of $\textrm{log }W'$ (by Theorem \ref{komponentit}, there are countably many of these).
For each $i$, consider $C_i \cap D$.
If $C_i \cap D=C_i$, then $C_i$ is an irreducible component of $D$.
If not, then $W_1=\textrm{exp}(C_i \cap D) \subsetneq W'$, and hence $\textrm{dim}(W_1) < \textrm{dim}(W')$.
Repeat now the process with $C_i \cap D$ in place of $D$, looking at the irreducible components of $W_1$.
Since the dimension drops at every step, the process eventually terminates.

This way, we obtain all the irreducible components of $D$.
At each step, we consider the finitely many irreducible components of some variety $X$ and the countably many irreducible components of $\textrm{log }X$.
Thus, $D$ has only countably many irreducible components.
 \end{proof}

\begin{remark}\label{ffff}
Let $D$ be a PQF-closed set, and let $C$ be an irreducible component of $D$.
Then, $\textrm{exp}(C) \subseteq \textrm{exp}(D)$, so 
$$\textrm{dim}(C)=\textrm{dim}(\textrm{exp}(C)) \le \textrm{dim}(\textrm{exp}(D)).$$ 
\end{remark}

\begin{definition}\label{dimgendef}
For an arbitrary PQF-closed set $C$, we define $\textrm{dim}(C)$ to be the maximum dimension of the irreducible components of $C$ (by Remark \ref{ffff}, this is a finite number). 

For any PQF-closed set $C$, we say that an element $a \in C$ is \emph{generic} if $\textrm{rk}(a/ \textrm{log }F_0)=\textrm{dim}(C)$, where $F_0$ is the smallest algebraically closed subfield of $F$ such that $C$ is definable over $\textrm{log}(F_0)$.
\end{definition}
 
It is now easy to see that for an arbitrary PQF-closed set $D$, it holds that $\textrm{dim}(D)=\textrm{dim}(\textrm{exp}(D))$.

Also, the  Dimension Theorem (i.e. axiom (Z3) of Zariski geometries) holds on the cover.
To see that this is the case, we need the following lemma.

\begin{lemma}\label{dimthmlemma}
Let $C_1, C_2 \subseteq V^n$ be closed and irreducible, and let $X$ be an irreducible component of $C_1 \cap C_2$.
Then, $\textrm{exp}(X)$ is an irreducible component of $\textrm{exp}(C_1) \cap \textrm{exp}(C_2)$.
\end{lemma}
 
\begin{proof}
For $i=1,2$, denote
$C_i=L_i \cap m_i \textrm{log } W_i$, where $L_i$ is linear, $W_i$ does not branch.
By Corollary \ref{olennainen} and Theorem \ref{komponentit}, we may assume that $T_i=\textrm{exp}(L_i)$ is the minimal torus containing $\textrm{exp}(C_i)=T_i \cap W_i^{m_i}$ and that $W_i^{m_i}$ is chosen as in the proof of Lemma \ref{leikkaus}.
Now,
$$\textrm{exp} (C_1) \cap \textrm{exp}(C_2)=T_1 \cap T_2 \cap W_1^{m_1} \cap W_2^{m_2}.$$
 
Let $W_1^{m_1} \cap W_2^{m_2}=T_3 \cap W$ where $T_3$ is a torus and $W$ is a variety not contained in any torus, chosen as in the proof of Lemma \ref{leikkaus}.
Set $T=T_1 \cap T_2 \cap T_3$ so that
 $$T_1 \cap T_2 \cap W_1^{m_1} \cap W_2^{m_2}=T \cap W.$$
 Let $X_1, \ldots, X_r$ be the irreducible components of $T\cap W$.
 For each $i$, we may write $X_i= T \cap T_i' \cap Y_i$, where $T_i'$ is the minimal torus containing $X_i$ and $Y_i$ is not contained in any torus, again chosen as in the proof of Lemma \ref{leikkaus}.
 Let $L_3$ be a linear set so that $\textrm{exp}(L_3)=T_3$ and $L_1 \cap L_2 \cap  L_3 \neq \emptyset$.  
 By Theorem \ref{komponentit}, the irreducible components of $\textrm{log}(T\cap W)$ are the sets
 $$((L \cap L_i')+k) \cap m_i' \textrm{log } {Y_i^{\frac{1}{m_i'}}}_{(j)},$$
 where $L=L_1 \cap L_2 \cap L_3$, and $L_i'$ is a linear set  such that $\textrm{exp}(L_i')=T_i'$ and $L_i' \cap L \neq \emptyset$, and $m_i'$ is such that the $m_i'$:th roots of $Y_i$ no longer branch.
 
Let $W_1^{\frac{1}{m_2}}$ and $W_2^{\frac{1}{m_1}}$ be the unique $m_2$:th and $m_1$:th roots of $W_1$ and $W_2$, respectively.
Then,
\begin{eqnarray*}
C_1 \cap C_2 &=& L_1 \cap L_2 \cap m_1 m_2 \textrm{log}(W_1^{\frac{1}{m_2}} \cap W_2^{\frac{1}{m_1}}),
\end{eqnarray*}
where
$$W_1^{\frac{1}{m_2}} \cap W_2^{\frac{1}{m_1}}=(W_1^{m_1} \cap W_2^{m_2})^{\frac{1}{m_1 m_2}}=(T_3 \cap W)^{\frac{1}{m_1 m_2}}$$
for a suitable choice of the $m_1 m_2$:th root.
Thus, for suitable choices of the $m_1 m_2$:th roots,
\begin{eqnarray*}
C_1 \cap C_2 &=& L_1 \cap L_2 \cap m_1 m_2 \textrm{log}((T_3 \cap W)^{\frac{1}{m_1 m_2}}) \\
&=& L_1 \cap L_2 \cap m_1 m_2 \textrm{log}((T_1 \cap T_2 \cap T_3 \cap W)^{\frac{1}{m_1 m_2}}) \\
&=&L_1 \cap L_2 \cap m_1 m_2 \textrm{log}(T^{\frac{1}{m_1 m_2} } \cap \bigcup_{i=1}^r (T_i'^{\frac{1}{m_1 m_2}}\cap Y_i^{\frac{1}{m_1 m_2}})).
\end{eqnarray*}
Hence, by Theorem \ref{komponentit}, irreducible components of $C_1 \cap C_2$ are of the form
$$((L_1 \cap L_2 \cap L_3 \cap L_i') +k) \cap m_i' \textrm{log } {Y_i ^{\frac{1}{m_i'}}}_{(j)},$$
and thus each one of them is an irreducible component of $\textrm{log}(T\cap W)$.
(Note that for each $i$, we have $T_i' \cap Y_i \subseteq T \cap T_i' \cap Y_i$.)
By Corollary  \ref{olennainen}, $\textrm{exp}(X)$ is an irreducible component of $T\cap W=\textrm{exp}(C_1 \cap C_2)$, as wanted.
\end{proof} 
 
\begin{theorem}[Dimension Theorem]\label{dimthm}
Let $C_1, C_2 \subset V^n$ be closed and irreducible.
Let $X$ be a non-empty irreducible component of $C_1 \cap C_2$.
Then, $\textrm{dim}(X) \ge \textrm{dim}(C_1)+\textrm{dim}(C_2)-n$.
\end{theorem}
\begin{proof}
By Lemma \ref{dimthmlemma}, $\textrm{exp}(X)$ is an irreducible component of $\textrm{exp}(C_1) \cap \textrm{exp}(C_2)$. 
Since the theorem holds on the field, the theorem follows from Lemma \ref{dimensio}.
\end{proof}
   
\section{The cover as a quasiminimal structure}

We wish to apply model theoretic methods developed in \cite{lisuriart}, section 2 (see also \cite{lisuri}, Chapter 2) to the cover.
For this, we need to obtain an abstract elementary class (AEC) satisfying certain properties.
Quasiminimal classes, studied in \cite{monet} and \cite{kir}, provide one such context.
Thus, in this section, we take a look at the cover as a quasiminimal structure.
Also, at the end, we show that the topological dimension defined in the previous section coincides with the dimension obtained from the following pregeometry (and in particular, the notion of generic elements given in Definition \ref{dimgendef} coincides with the usual notion of ``generic" in stability theory).
This will be used in the arguments of section 5. 

\begin{definition}
Define a closure operator, cl, on $\mathcal{P}(V)$ so that for any $A \subset V$, 
$$\textrm{cl}(A)=\textrm{log}(\textrm{acl}(\textrm{exp}(A))).$$
\end{definition} 

It is easy to see that $\mathcal{P}(V)$ forms a pregeometry with respected to cl.
After adding constants for the elements of $\textrm{log}(\overline{\mathbb{Q}})$ to our language (here $\overline{\mathbb{Q}}$ stands for the field of algebraic numbers), it can be proved (\cite{Zilber}, section 3) that $V$ has become  a quasiminimal pregeometry structure in the sense of \cite{monet}, defined below (here, tp denotes quantifier free type).

\begin{definition}\label{quasiminclass}
Let $M$ be an $L$-structure for a countable language $L$, equipped with a pregeometry $cl$.
We say that $M$ is a \emph{quasiminimal pregeometry structure} if the following hold:
\begin{enumerate}
\item (QM1) The pregeometry  is determined by the language.
That is, if $a$ and $a'$ are singletons and
$\textrm{tp}(a, b)=\textrm{tp}(a', b')$, then $a \in \textrm{cl}(b)$ if and only if $a' \in \textrm{cl}(b')$.
\item (QM2) $M$ is infinite-dimensional with respect to cl.
\item (QM3) (Countable closure property) If $A \subseteq M$ is finite, then $\textrm{cl}(A)$ is countable.
\item (QM4) (Uniqueness of the generic type) Suppose that $H,H' \subseteq M$ are countable closed subsets, enumerated so that $\textrm{tp}(H)=\textrm{tp}(H')$.
If $a \in M \setminus H$ and $a' \in M \setminus H'$ are singletons, then $\textrm{tp}(H,a)=\textrm{tp}(H',a')$ (with respect to the same enumerations for $H$ and $H'$).
\item (QM5) ($\aleph_0$-homogeneity over closed sets and the empty set)
Let $H,H' \subseteq M$ be countable closed subsets or empty, enumerated so that $\textrm{tp}(H)=\textrm{tp}(H')$,
and let $b, b'$ be finite tuples from $M$ such that $\textrm{tp}(H, b)=\textrm{tp}(H',b')$, and let $a$ be a singleton such that $a \in \textrm{cl}(H, b)$.
Then there is some singleton $a' \in M$ such that $\textrm{tp}(H, b,a)=\textrm{tp}(H',b',a').$
\end{enumerate}
\end{definition}

From now on, we will assume that we have added the elements of $\textrm{log}(\overline{\mathbb{Q}})$ to our language. 
We can construct an AEC $\K(V)$ from some  model $V$ of the theory of the covers (see \cite{monet}, section 2, and the end of section 2 in \cite{lisuriart}).
We may assume $V$ is sufficiently large (in particular, uncountable), and thus, in terms of model theory, a monster model for the class.
Using techniques applicable for AECs we obtain for $\mathcal{K}(V)$ an independence calculus that satisfies all the usual properties of non-forking (\cite{lisuriart}, section 2, see also \cite{lisuri}, chapter 2).
For quasiminimal pregeometry structures, it then turns out that the independence notion coincides with the one given by the pregeometry.
In the case of the cover, we will get the same notion also from the dimension in the PQF-topology.
We write $a \da_b c$ for ``$a$ is independent from $c$ over $b$".

We will be using the concepts of Galois type and bounded closure.
In \cite{lisuriart}, section 2 (see also \cite{lisuri}, section 2), we showed that these concepts are meaningful in our context.

\begin{definition}
We define \emph{Galois types} as orbits of automorphisms of $V$.
For $A \subset V$, we write $t^g(b/A)=t^g(c/A)$ if the tuples $b$ and $c$ have the same Galois type over $A$, i.e. if there is some automorphism $f$ of $V$ such that $f(a)=a$ for every $a \in A$ and $f(b)=c$.
\end{definition}

\begin{remark}\label{le31}
By (e.g.) Lemma 3.1 in \cite{monet}, the quantifier-free first order types imply Galois types over closed sets and finite sets.
\end{remark}

\begin{definition}
We say a set $X \subset V^n$ is \emph{Galois definable} over $A \subset V$ for every automorphism $f$ of $V$ such that $f(a)=a$ for every $a \in A$, it holds that $f(X)=X$.
\end{definition}

\begin{definition}
We say that a set $A$ is \emph{bounded} if $\vert A \vert <  \vert V \vert$.
\end{definition}

\begin{definition}
We say $a$ is in the \emph{bounded closure} of $A$, denoted $a \in \textrm{bcl}(A)$, if $t^g(a/A)$ has only boundedly many realizations.
\end{definition}

\begin{lemma}\label{bcl} 
For any $A \subset V$, 
$$\textrm{bcl}(A)=\textrm{cl}(A).$$
\end{lemma}

\begin{proof}
Clearly $\textrm{cl}(A)=\textrm{log}(\textrm{acl}(\textrm{exp}(A))) \subseteq \textrm{bcl}(A)$.
 
On the other hand, suppose $v, v' \notin \textrm{cl}(A).$
We will show that $v$ and $v'$ have the same Galois type over $\textrm{cl}(A)$. 
By (QM4) and Remark \ref{le31},  $t^g(v/\textrm{cl}(A)))=t^g(v'/\textrm{cl}(A))$.
Since there are uncountably many such $v'$, we have $v \notin \textrm{bcl}(A)$.
\end{proof}
 
\begin{remark}\label{ranksama}
The dimension obtained from the pregeometry agrees with the one defined topologically with respect to the PQF-closed sets, i.e. for any $a \in V^n$, $A \subset V$, it holds that $$\textrm{rk}(a/A)=\textrm{dim}_{\textrm{cl}}(a/A).$$ 
To see this, let $X \subset V^n$ be the smallest $A$-definable PQF-closed set containing $a$.
Then, $\textrm{exp}(X)$ is definable by some first-order formula $\phi(x_1, \ldots, x_n)$ in the field language over $\textrm{exp}(A)$.
By Corollary \ref{variety},  the set $\textrm{exp}(X)$ is Zariski closed, and thus there is some set of polynomials $S=\{p_1, \ldots, p_m\}$ such that $\textrm{exp}(X)$ is the zero locus of $S$.

Let $a_1, \ldots, a_r \in F$ be the coefficients of the polynomials $p_1, \ldots, p_m$ that are in $F \setminus \textrm{acl}(\textrm{exp}(A))$.
Replacing these by the variables $y_1, \ldots, y_r$, we may write a formula in the field  language with parameters from $\textrm{exp}(A)$ expressing
$$\exists y_1 \cdots \exists y_r (\phi(x_1, \ldots, x_n) \leftrightarrow \bigwedge_{i=1}^m p_i(y_1, \ldots, y_r, x_1, \ldots, x_n)=0).$$
This formula holds in $F$.
By Model Completeness in algebraically closed fields, it holds already in 
$\textrm{acl}(\textrm{exp}(A))$.  
Thus there are some polynomials $p_1', \ldots, p_m'$ over $\textrm{acl}(\textrm{exp}(A))$ such that $\textrm{exp}(X)$ is the zero locus of the set consisting of these polynomials.

Hence, we have seen that $\textrm{exp}(X)$ is definable as a variety over $\textrm{acl}(\textrm{exp}(A))$, the smallest algebraically closed field containing $A$.
This implies that $\textrm{exp}(X)$ is the locus (in the field sense) of $\textrm{exp}(a)$ over $\textrm{acl}(\textrm{exp}(A))$, so
$$\textrm{dim}(X)=\textrm{dim}(\textrm{exp}(X))=\textrm{dim}_{\textrm{acl}}(\textrm{exp}(a)/\textrm{acl}(\textrm{exp}(A)))=\textrm{dim}_{cl}(a/A),$$
where the second equality follows form the fact that in algebraically closed fields, the dimension in Zariski topology, Morley rank, and the dimension with respect to the pregeometry defined by $acl$ are all equal.
 
It also follows that the definition of generic elements of a PQF-closed set given in \ref{dimgendef} coincides with the usual notion of ``generic" in stability theory.
\end{remark}

From now on, we will use $\textrm{dim}$ to denote the pregeometry dimension $\textrm{dim}_{cl}$.
We write $\textrm{dim}(a)$ for $\textrm{dim}(a/\emptyset)$. 
  
\section{The cover is Zariski-like}

In this section we show that the cover is Zariski-like in sense of \cite{lisuriart}, Definition 4.1 (see also \cite{lisuri}, Definition 4.1).
Before we can list the axioms for a Zariski-like structure, we need some auxiliary definitions.
First, we generalize the notion of specialization from \cite{HrZi} to our context.

\begin{definition}
Let $\M$ be a monster model for a quasiminimal class, $A \subset \M$, and let $\mathcal{C}$ be a collection of subsets of $\M^n$, $n=1,2,\ldots$.
We say that a function $f: A \to \M$ is a \emph{specialization} (with respect to $\mathcal{C}$) if for any $a_1, \ldots, a_n \in A$ and for any $C \in \mathcal{C}$, it holds that if $(a_1, \ldots, a_n) \in C$, then $(f(a_1), \ldots, f(a_n)) \in C$.
If $A=(a_i : i \in I)$, $B=(b_i:i\in I)$ and the indexing is clear from the context, we write $A \to B$ if the map $a_i \mapsto b_i$, $i \in I$, is a specialization.

If $a$ and $b$ are finite tuples and $a \to b$, we denote $\textrm{rk}(a \to b)=\textrm{dim}(a/\emptyset)-\textrm{dim}(b/\emptyset)$.
\end{definition}

The specializations in the context of Zariski geometries in \cite{HrZi} are specializations in the sense of our definition if we take $\mathcal{C}$ to be the collection of closed sets (Zariski geometries are quasiminimal since they are strongly minimal).

Next, we generalize the definition of regular specializations from \cite{HrZi}.  

\begin{definition}
Let $\M$ be a monster model for a quasiminimal class.
We define a \emph{strongly regular} specialization recursively as follows:
\begin{itemize}
\item Isomorphisms are strongly regular;
\item If $a \to a'$ is a specialization and $a \in \M$ is generic over $\emptyset$, then $a \to a'$ is strongly regular;
\item $aa' \to bb'$ is strongly regular if $a \downarrow_\emptyset a'$ and the specializations $a \to b$ and $a' \to b'$ are strongly regular.
\end{itemize}
\end{definition}

\begin{remark}
It follows from Assumptions 6.6 (7) in \cite{HrZi} (for a more detailed discussion on why these properties hold in a Zariski geometry, see \cite{lisuri}, Chapter 1.1.) that if a specialization on a Zariski geometry is strongly regular in the sense of our definition, then it is regular in the sense of \cite{HrZi} (definition on p. 25).
  
Also, it is easy to see that if $a \to a'$ is a strongly regular specialization on the cover, then $\textrm{exp}(a) \to \textrm{exp}(a')$ is strongly regular on the field and thus regular in the sense of \cite{HrZi}.
\end{remark}
 
The following generalizes the definition of good specializations from \cite{HrZi}.

\begin{definition}\label{good}
We define a \emph{strongly good} specialization recursively as follows.
Any strongly regular specialization is strongly good.
Let $a=(a_1, a_2, a_3)$, $a'=(a_1', a_2', a_3')$, and $a \to a'$.
Suppose:
\begin{enumerate}[(i)]
\item $(a_1, a_2) \to (a_1', a_2')$ is strongly good.
\item $a_1 \to a_1'$ is an isomorphism.
\item $a_3 \in \textrm{cl}(a_1)$.
\end{enumerate}
Then, $a \to a'$ is strongly good.
\end{definition}

A Zariski-like structure is defined by nine axioms as follows.  

\begin{definition}
We say that an infinite-dimensional quasiminimal pregeometry structure $\M$ is Zariski-like if there is a countable collection $\mathcal{C}$ of subsets of $\M^n$ ($n=1, 2, \ldots$), which we call the \emph{irreducible} sets, satisfying the following axioms (all specializations are with respect to $\mathcal{C}$).
 
\begin{enumerate}[(1)]
\item Each $C \in \mathcal{C}$ is Galois definable over $\emptyset$.
\item For each $a \in \M$, there is some $C \in \mathcal{C}$ such that $a$ is generic in $C$, i.e. of all elements in $C$, $a$ has the maximal dimension over $\emptyset$.  
\item If $C \in \mathcal{C}$, then the generic elements of $C$ have the same Galois type. 
\item If $C,D \in \mathcal{C}$, $a \in C$ generic and $a \in D$, then $C \subseteq D.$
\item If $C_1, C_2 \in \mathcal{C}$, $(a,b) \in C_1$ is generic, $a$ is a generic element of $C_2$ and $(a',b') \in C_1$, then $a' \in C_2$.
\item If $C \in \mathcal{C}$, $C \subset \M^n$, and $f$ is a coordinate permutation on $\M^n$, then $f(C) \in \mathcal{C}$.
\item Let $a \to a'$ be a strongly good specialization  and let $rk(a \to a') \le 1$.
Then any specializations $ab \to a'b'$, $ac \to a'c'$  can be amalgamated: there exists $b^{*}$, independent from $c$ over $a$ such that $\textrm{t}^g(b^{*}/a)=\textrm{t}^g(b/a)$, and $ab^{*}c \to a'b'c'$.
\item Let $(a_i:i\in I)$ be independent and indiscernible over $b$.
Suppose $(a_i':i\in I)$ is indiscernible over $b'$, and $a_ib \to a_i'b'$ for each $i \in I$.
Further suppose $(b \to b')$ is a strongly good specialization  and $rk(b \to b') \le 1$.
Then, $(ba_i:i \in I) \to (b'a_i':i\in I)$.

\item Let $\kappa$ be a (possibly finite) cardinal and let  $a_i, b_i \in \M$ with $i < \kappa$, such that $a_0 \neq a_1$ and $b_0=b_1$. 
Suppose $(a_i)_{i<\kappa} \to (b_i)_{i <\kappa}$ is a specialization.
Assume there is some unbounded and directed $S \subset \mathcal{P}_{<\omega}(\kappa)$ satisfying the following conditions:
\begin{enumerate}[(i)]
\item  $0,1 \in X$ for all $X \in S$;
\item For all $X,Y \in S$ such that $X \subseteq Y$,  and for all sequences $(c_i)_{i \in Y}$ from $V$, the following holds: 
If $c_0=c_1$, $ (a_i)_{i \in Y} \to (c_i)_{i \in Y} \to (b_i)_{i \in Y},$
and $\textrm{rk}((a_i)_{i \in Y} \to (c_i)_{i \in Y}) \le 1$, then $\textrm{rk}((a_i)_{i \in X} \to (c_i)_{i \in X}) \le 1$.
\end{enumerate}
 
Then, there are $(c_i)_{i<\kappa}$ such that   
$$(a_i)_{i \in \kappa} \to (c_i)_{i \in \kappa} \to (b_i)_{i \in \kappa},$$
$c_0=c_1$ and $\textrm{rk}((a_i)_{i \in X} \to (c_i)_{i \in X}) \le 1$ for all $X \in S$.
 \end{enumerate}
\end{definition}

\begin{remark}
Zariski geometries are Zariski-like if the collection $\mathcal{C}$ is taken to be the irreducible closed sets definable over $\emptyset$.
Indeed, strongly minimal structures are quasiminimal, and the conditions (1)-(9) are satisfied.
On a Zariski geometry, first-order types imply Galois types.
Moreover, every strongly regular specialization is regular, and every strongly good specialization is good.
Hence, (7) is Lemma 5.14 in \cite{HrZi} and (8) is Lemma 5.15 in \cite{HrZi}.
(9) holds by Compactness.
\end{remark}

As in Chapter 4, we assume that we have added elements of $\textrm{log}(\overline{\mathbb{Q}})$ to our language and that we are working in a monster model $V$ for the theory of covers.
Sets definable by positive, quatifier-free formulae over $\emptyset$ will be called
\emph{PQF-closed over $\emptyset$}.  
We will show that $V$ is Zariski-like if we take $\mathcal{C}$ to consist of the irreducible PQF-closed  sets definable over $\emptyset$.
Since we have added the elements of $\textrm{log}(\overline{\mathbb{Q}})$ as constants in our language,  this means in practice that we are allowed to use parameters from $\textrm{cl}(\emptyset)$.  

In Chapter 4, we saw that $V$ is a quasiminimal pregeometry structure.
Definable sets are Galois definable, so (1) holds.
For (2), if $W$ is the locus of $\textrm{exp}(v)$ over $\overline{\mathbb{Q}}$, take $C$ to be the irreducible component of $\textrm{log }W$ containing $v$.
For (4), we note that by Remark \ref{ranksama}, $C$ must be the locus of $a$.
(5) and (6) follow straight from the definition of a $PQF$-closed set.

\begin{remark}\label{irredykskas}
Let $W$ be a Zariski closed set on the field $F$.
Note that if $W$ is irreducible on some algebraically closed subfield $F' \subset F$ (e.g. over $\overline{\mathbb{Q}}$), then it is irreducible also on $F$.
This follows from Model Completeness for algebraically closed fields or from the fact that if an ideal $\langle f_1, \ldots, f_r \rangle \subset F'[x_1, \ldots, x_n]$ is prime, then also the ideal $\langle f_1, \ldots, f_r \rangle \subset F[x_1, \ldots, x_n]$ is prime.
\end{remark}
 
\begin{remark}\label{kuntapqf}
Let $C$ be a PQF-closed subset of the cover, definable (by a positive quantifier-free formula) over $A \subset V$.
By Corollary \ref{variety}, $\textrm{exp}(C)$ is a variety.
By Corollary \ref{variety} and Model Completeness, $\textrm{exp}(C)$ is definable by polynomials with coefficients in the smallest algebraically closed field containing $A$ (see the arguments in Remark \ref{ranksama}).
Thus, in particular, if $C$ is definable over $\emptyset$, then $\textrm{exp}(C)$ is definable by polynomials with coefficients in $\overline{\mathbb{Q}}$.
\end{remark} 

We now prove (3).  

\begin{lemma}\label{samegaloistype}
The generic elements of an irreducible set PQF-closed over $\emptyset$ have the same Galois type over $\emptyset$.
\end{lemma}

\begin{proof}
Let $C \subset V^n$ be irreducible and PQF-closed over $\emptyset$, $a=(a_1, \ldots, a_n), b=(b_1, \ldots, b_n) \in C$ generic.
We may without loss assume that there is some $k \le n$ such that $a_1, \ldots, a_k$ are linearly independent and $a_{k+1}, \ldots, a_n \in \textrm{span}(a_1, \ldots, a_k)$.
Let $L$ be the linear set given by the equations expressing this.
Then, we have $a \in L$, and thus, as $a$ is generic on $C$, also $C \subseteq L$.
Hence, $b \in L$, so $b_{k+1}, \ldots, b_n \in \textrm{span}(b_1, \ldots, b_k)$.
In particular, $a_{k+1}, \ldots, a_n \in \textrm{bcl}(a_1, \ldots, a_k)$ and $b_{k+1}, \ldots, b_n \in \textrm{bcl}(b_1, \ldots, b_k)$.
Thus, by Lemma \ref{bcl}, 
$$\textrm{dim}(a_1, \ldots, a_k)=\textrm{dim}(a)=\textrm{dim}(b)=\textrm{dim}(b_1, \ldots, b_k).$$

Denote $a'=(a_1, \ldots, a_k)$ and $b'=(b_1, \ldots, b_k)$.
Let $D$ be the locus of $a'$.
Since $a$ is generic on $C$, we have $b' \in D$, and thus also $b'$ is generic on $D$ as it has the same dimension as $a'$.
Since there are no linear dependencies between the elements $a_1, \ldots, a_k$, we have $D=m \textrm{log } W$ for some variety $W$ that does not branch.
Then, $W^m$ is the locus of $\textrm{exp}(a')$.
Let $l \in \mathbb{N}$ be arbitrary.
We have $\textrm{exp}(\frac{a'}{m}) \in W.$
Since $W$ does not branch, it has a unique $l$:th root, $W^{\frac{1}{l}}$.
Now, $\frac{a'}{ml} \in \frac{1}{l} \textrm{log } W$, and thus $\textrm{exp}(\frac{a'}{ml}) \in W^{\frac{1}{l}}$.
Then,
$$\textrm{exp}\left(\frac{a'}{l} \right) = \left( \textrm{exp} \left( \frac{a'}{lm}\right)\right)^m \in (W^{\frac{1}{l}})^m,$$
where $(W^{\frac{1}{l}})^m$ is a single $l$:th root of $W^m$.  
Thus, for each $l$, we are able to determine in which $l$:th root of $W^m$ the element $\textrm{exp}(\frac{a'}{l})$ lies. 
Since $W^m$ is the locus of $\textrm{exp}(a')$, we have $\textrm{exp}(a') \notin Y$ for all $Y \subset W$ such that $\textrm{dim}(Y)<\textrm{dim}(W)$, and since $a_1, \ldots, a_k$ are linearly independent, $za' \neq 0$ for every $z \in \mathbb{Z}^n \setminus \{0\}$.
Hence, by Lemma \ref{tyypit}, the set $D$ determines the quantifier-free type of $a'$ over $\emptyset$.
The element $b'$ is also generic on $D$, so the same argument applies to it.
By \cite{monet}, Lemma 3.1 and \cite{kir}, Theorem 3.3 and QM5, quantifier free types determine Galois types. 
Hence $t^g(a'/\emptyset)=t^g(b'/\emptyset)$.
 
There is an automorphism taking $a'=(a_1, \ldots, a_k)$ to $b'=(b_1, \ldots, b_k)$.
Since the elements $b_{k+1}, \ldots, b_n$ are determined from the elements $b_1, \ldots, b_k$ by exactly the same linear equations that determine the elements $a_{k+1}, \ldots, a_n$ from $a_1, \ldots, a_k$, it must map $(a_{k+1}, \ldots, a_n)$ to $(b_{k+1}, \ldots, b_n)$, and hence $t^g(a/\emptyset)=t^g(b/\emptyset)$. 
 \end{proof}

The proof of (7) is an adaptation of the proof of Lemma 5.14 in \cite{HrZi}.  
For it, we will need the following definition. 
  
\begin{definition}
Let $C \subset V^{n+m}$ be an irreducible set.
We say an element $a \in V^n$ is \emph{good} for $C$ if there is some $b \in V^m$ so that $(a,b)$ is a generic element of $C$.
\end{definition}  
 
We now present a couple of lemmas needed for (7).
 
\begin{lemma}\label{poikkeuspisteet}
Suppose $C \subset V^{n+m}$ is irreducible, $a \in V^n$ is good for $C$, $a \to a'$ and $rk(a \to a') \le 1$.
Then, $\textrm{dim}(C(a')) \le \textrm{dim}(C(a))$.
\end{lemma}

\begin{proof}
We give the argument in the case that $C$ is PQF-closed over $\emptyset$.
If it is not, then there are some parameters needed in defining $C$, and we have to take them into account in the calculations that will follow.
However, the calculations will remain similar to the ones we present here.  

Suppose for the sake of contradiction that $r_1=\textrm{dim}(C(a'))> \textrm{dim} (C(a))=r_2$.
It follows from the assumptions that on the field, $\textrm{exp}(a) \to \textrm{exp}(a')$ and $\textrm{rk}(\textrm{exp}(a) \to \textrm{exp}(a')) \le 1$.
Since algebraically closed fields are Zariski geometries, \cite{HrZi}, Lemma 4.12 gives   
\begin{eqnarray}\label{aboveinequality}
\textrm{dim}(\textrm{exp}(C)(\textrm{exp}(a'))) \le \textrm{dim}(\textrm{exp}(C)(\textrm{exp}(a))).
\end{eqnarray}

Let $b' \in C(a')$ be such that $\textrm{dim}(b'/a')=r_1$.
Then, $\textrm{exp}(b') \in \textrm{exp}(C)(\textrm{exp}(a'))$, and by Lemma
\ref{dimensio},  $\textrm{dim}(\textrm{exp}(C)(\textrm{exp}(a'))) \ge r_1$.
By (\ref{aboveinequality}), there is some $x \in \textrm{exp}(C)(\textrm{exp}(a))$ such that $\textrm{dim}(x/\textrm{exp}(a)) \ge r_1$, and thus some $k \in K^n$, $b \in V^m$ such that $(a+k, b) \in C$ and $x=\textrm{exp}(b)$.
Since $\textrm{dim}(a+k)=\textrm{dim}(a)$, we have
$$dim(b/a+k) \le \textrm{dim}(C)-\textrm{dim}(a)=r_2,$$
a contradiction since we also have 
$$\textrm{dim}(b/a+k)=\textrm{dim}(\textrm{exp}(b)/\textrm{exp}(a))\ge r_1>r_2.$$
\end{proof}

\begin{lemma}\label{henkitore}
Suppose $C$ is irreducible and $a$ is good for $C$.
Suppose $(C_i)_{i<\omega}$ is a collection of irreducible sets such that it is permuted by all automorphisms of $V$.
Let $b$ be an element such that $ab$ is good for each $C_i$ and that for any $c \in C(a)$ generic over $b$
it holds that $c \in C_i(ab)$ for some $i$.
Assume $(a, b) \to (a', b')$, $a \to a'$ is a strongly regular specialization, $\textrm{rk}(a \to a') \le 1$,
and $c'$ is such that $(a',c') \in C$.

Then, there is some $i <\omega$ so that $(a',b',c') \in C_i$.
\end{lemma}

\begin{proof} 
Let $D$ be the locus of $(a, b)$, and let $m=\textrm{lg}(c)$.
Denote 
$$C^{*}=\{(x,y,z) \in D \times V^m \, | \, (x,z) \in C\}.$$
%and
%$$C_2^*=\{(x, y,z) \in D \times V^m \, | \, (x,z) \in C, \, (x,y,z) \in C_i \textrm{ for some }  i <\omega\}.$$
%(Note that $C_2^*$ is not definable.)
Let $E$ be the locus of $a$.%  

We first show that the irreducible components of $C^*$ not contained in
$$X=\{(x,y,z) \in D \times V^m \, | \, \textrm{dim}(a)-\textrm{dim}(x)>1 \textrm{ or $\textrm{exp}(x)$ is not regular on $\textrm{exp}(E)$}\}$$
all have same dimension $\textrm{dim}(D)+\textrm{dim}(c/a)$ (``regular" is in the sense of \cite{HrZi}, defined on p.20).
(Note that the set $X$ is not definable.)
Suppose $Y$ is some such irreducible component and
$(d,e,f) \in Y$  generic. 
Then, we have $\textrm{dim}(d) \ge \textrm{dim}(a)-1$, and thus by Lemma \ref{poikkeuspisteet}, $\textrm{dim}(f/d) \le \textrm{dim}(c/a)$.
%Huomaa, ettŠ tuossa jos geneerisen pisteen ekan koordinaatin rankki ois tota pienempi, niin sama pŠtee koko komponentille, ne on samassa sulj joukossa!
Hence, 
$$\textrm{dim}(d,e,f) \le \textrm{dim}(D)+\textrm{dim}(c/a).$$
 
To prove $\textrm{dim}(Y) \ge  \textrm{dim}(D)+\textrm{dim}(c/a)$, we note that in the PQF-topology, $C^*$ is isomorphic to $(D \times C) \cap (\Delta_E \times V^{n+m})$, where $n=\textrm{lg}(b)$ and 
$\Delta_E=\{(x,y) \in E \times E \, | \, x=y \}$.
Both $D \times C$ and $\Delta_E \times V^{n+m}$ are irreducible.
By Theorem \ref{dimthm},  $\textrm{exp}(Y)$ is an irreducible component of $\textrm{exp}(C^*)$.

We obtain a copy of  $\textrm{exp}(C^*)$ by intersecting $\textrm{exp}(D) \times \textrm{exp}(C)$ with a suitable $\textrm{exp}(E)$-diagonal.
Indeed, if we have $(x, y, c_1, c_2) \in \textrm{exp}(D) \times \textrm{exp}(C)$, where $(x, y) \in \textrm{exp}(D)$ and $(c_1, c_2) \in \textrm{exp}(C)$, then the diagonal $\Delta$ expressing ``$x=c_1$" is as wanted.  
Since $Y$ is not contained in $X$, the isomorphic copy of $\textrm{exp}(Y)$ obtained in the above procedure contains some point $(c_1, y, c_1, c_2)$ where $c_1$ is a regular point of $\textrm{exp}(E)$.
Thus, using Lemma 5.4 in \cite{HrZi} to calculate dimensions in $W=\textrm{exp}(E) \times F^n \times \textrm{exp}(E) \times F^m$ and Lemma \ref{dimensio}, we obtain
\begin{eqnarray*}
\textrm{dim}(Y) &\ge& \textrm{dim}(D)+\textrm{dim}(C)+\textrm{dim}(\Delta_E)+n+m-\textrm{dim}(W) \\
&=& \textrm{dim}(D)+\textrm{dim}(C)-\textrm{dim}(E)\\
&=& \textrm{dim}(D)+\textrm{dim}(c/a),
\end{eqnarray*}
since $\textrm{dim}(W)=2\textrm{dim}(E)+m+n$, $\textrm{dim}(\Delta_E)=\textrm{dim}(E)$ and
$\textrm{dim}(C)-\textrm{dim}(E)=\textrm{dim}(c/a)$.
 
Let $C'$ be the irreducible component of $C^*$ containing $(a', b',c')$.
Then, $C'$ is not included in $X$.  
Indeed, $\textrm{rk}(a \to a') \le 1$ and the fact that $a \to a'$ is strongly regular implies that $\textrm{exp}(a) \to \textrm{exp}(a')$ is regular, i.e. $\textrm{exp}(a')$ is regular on the locus of $\textrm{exp}(a)$ which is
$\textrm{exp}(E)$. 
Let $(d,e,f)$ be a generic point of $C'$.
Then, $(d,f) \in C$ and $\textrm{dim}(d) \ge \textrm{dim}(a)-1$. 
By Lemma \ref{poikkeuspisteet}, $\textrm{dim}(f/d) \le \textrm{dim}(c/a)$, so $\textrm{dim}(f/d,e) \le \textrm{dim}(c/a)$.
As $\textrm{dim}(d,e,f)=\textrm{dim}(D)+\textrm{dim}(c/a)$, this implies that $(d,e)$ is a generic point of $D$.
It also follows that $\textrm{dim}(c/a)=\textrm{dim}(f/d,e)=\textrm{dim}(f/d)$, so in particular $f \da_d e$.
There is some automorphism taking $(d,e) \mapsto (a,b)$.
This automorphism then takes $f$ to some element $c \in C(a)$ such that $c \da_a b$.
By our assumptions, $(a,b,c) \in C_i$ for some $i<\omega$.
Since automorphisms permute the collection of the sets $C_i$, we have that $(d,e,f) \in C_j$ for some $j< \omega$.
Hence, $C' \subset C_j$, so in particular $(a',b',c') \in C_j$.
\end{proof}

The following is an analogue of Proposition 4.7 in \cite{HrZi}.

\begin{lemma}\label{prop4.7}
Let $C$, $D$ be irreducible, $a \in D$ generic, and suppose $a$ is good for $C$.
 Let $r=\textrm{dim}(C(a))\ge 0$.
Let $c$ and $c'$ be such that $(a,c) \to (a',c')$ and $a \to a'$ is strongly regular and $\textrm{rk}(a \to a') \le 1$.
Let $b' \in C(a')$.

Then, there exists $b \in C(a)$ such that $\textrm{dim}(b/ac)=r$ and $(a,b,c) \to (a',b',c')$.
 \end{lemma}

\begin{proof}
Let $E$ be the locus of $(a,c)$, $C^{*}=\{(x,z,y) \, | \, (x,z) \in E, (x,y) \in C\}$.
Let $C_i$, $i<\omega$ be the irreducible components of $C^*$ satisfying $\textrm{dim}(C_i(a,c))=r$.
We claim that
$$C(a) \subseteq \bigcup_{i<\omega} C_i(a,c).$$

We show first that every irreducible component of $C(a)$ is of dimension $r$. 
Let $m=\textrm{lg}(b')$.
Then, $C \subset D \times V^m$ and $C(a)=C \cap (\{a\} \times V^m)$.
Since $a$ is generic in $D$, $\textrm{exp}(a)$ is regular on $\textrm{exp}(D)$.
Applying Lemma 5.4 in \cite{HrZi} on $\textrm{exp}(D \times V^m)$,
we get that every nonempty irreducible component of $C(a)$ has dimension at least
$\textrm{dim}(C)+m-\textrm{dim}(D)-m=r.$

So in every such component, there is an element $x$ such that $\textrm{dim}(x/ca)=r$.
Then, $(a,c,x) \in C'$ for some irreducible component $C'$ of $C^*$, and thus $x \in \bigcup_{i<\omega} C_i(a,c)$.
Hence, the irreducible component of $C(a)$ containing $x$ is included in $\bigcup_{i<\omega} C_i(a,c)$.
We conclude that $C(a) \subseteq \bigcup_{i<\omega} C_i(a,c)$.
 
Now, by Lemma \ref{henkitore}, $b' \in C_i(a',c')$ for some $i$.
Let $b$ be a generic point of $C_i(a,c)$.
Since $C_i$ is irreducible and $(a,b,c)$ is a generic point of $C_i$, we have $(a,b,c) \to (a',b',c')$.
\end{proof} 

We are now ready to prove (7).

\begin{theorem}\label{amalg} 
Let $a \to a'$ be a strongly good specialization with $\textrm{rk}(a \to a') \le 1$.
Then any specializations $ab \to a'b'$, $ac \to a'c'$ can be amalgamated: there exists $b^{*}$, independent from $c$ over $a$ such that $\textrm{t}^g(b^{*}/a)=\textrm{t}^g(b/a)$, and $ab^{*}c \to a'b'c'$.
\end{theorem}
 
\begin{proof}
We prove the Lemma by induction, using Definition \ref{good}.
If $a \to a'$ is strongly regular, this follows from Lemma \ref{prop4.7}:
Let $D_1$ be the locus of $a$, $m=\textrm{lg}(b)$ and $C$ the locus of $(a,b)$ in $D_1 \times V^m$.
By Lemma \ref{prop4.7}, there is some $b^* \in C(a)$ such that $\textrm{dim}(b^*/ac)=\textrm{dim}(C(a))$ with $ab^*c \to a'c'b'$.
In particular, $\textrm{dim}(b^*/a)=\textrm{dim}(b/a)$, and $b^*$ is independent from $c$ over $a$.
Since $C$ is the locus of $(a,b)$, both $(a,b)$ and $(a,b^*)$ are generic on $C$.
By Lemma \ref{samegaloistype}, $t^g(ab/\emptyset)=t^g(ab^*/\emptyset)$, so in particular $\textrm{t}^g(b^*/a)=\textrm{t}^g(b/a)$.
%We may choose $c^*=c$.

Suppose now $a=(a_1, a_2,a_3)$, $a'=(a_1',a_2',a_3')$, $a \to a'$ are as in Definition \ref{good}, and the lemma holds for the specialization $(a_1, a_2) \to (a_1', a_2')$.
Amalgamating over $(a_1, a_2) \to (a_1', a_2')$, we see that there exist $b^*$, $a_3^*$ such that $\textrm{t}^g(b^* a_3^*/a_1 a_2)=\textrm{t}^g(ba_3/a_1 a_2)$, $b^*a_3^*$ is independent from $c$ over $a_1 a_2$, and $a_1 a_2 a_3^*b^*a_3 c \to a_1' a_2' a_3' b' a_3' c'$.
Now, by Definition \ref{good}, $a_3 \in \textrm{bcl}(a_1)$, and as $\textrm{t}^g(b^* a_3^*/a_1 a_2)=\textrm{t}^g(ba_3/a_1 a_2)$, also $a_3^* \in \textrm{bcl}(a_1)$.
By Definition \ref{good}, $a_1 \to a_1'$ is an isomorphism.
From these facts together it follows that $a_1a_3a_3^* \to a_1' a_3' a_3'$ is of rank $0$ and thus an isomorphism.
Hence, $a_3=a_3^*$, so we get $ab^*c \to a'b'c'$ as wanted.
\end{proof}

For (8) we need the concept of Morley and strongly indiscernible sequences.

\begin{definition}
Let $\A$ be a model.
We say a sequence $(a_{i})_{i<\a}$ is \emph{Morley} over $\mathcal{A}$,  
if for all $i<\a$, $t^g(a_{i}/\A )=t^g(a_{0}/\A )$ and
$a_{i}\da_{\A}\cup_{j<i}a_{j}$.
%For an arbitrary set $A$, we say that a sequence $(a_{i})_{i<\a}$ is \emph{Morley} over $A$, if there is some model $\mathcal{A}$ such that $A \subset \mathcal{A}$ and $(a_{i})_{i<\a}$ is Morley over $\mathcal{A}$.
 \end{definition}
 
\begin{definition}
We say that a sequence $(a_{i})_{i<\alpha}$ is \emph{strongly indiscernible} over $A$ if for all
cardinals $\kappa>\alpha$, there are $a_{i}$, $\alpha \le i<\kappa$, such that every permutation of the sequence
$(a_{i})_{i<\kappa}$ extends to an automorphism $f$ of $V$ such that $f(a)=a$ for every $a \in A$.
\end{definition}

We need the following facts about Morley sequences.

\begin{lemma}\label{morleyindisc}
If $\A$ is a countable model and the sequence $(a_{i})_{i<\alpha}$ is Morley over  $\A$, then it is independent over $\A$ and strongly indiscernible over $\A$.
\end{lemma}

\begin{proof}
By Lemma 2.38 in \cite{lisuriart} (Lemma 2.43 in \cite{lisuri}), our definition of a Morley sequence is equivalent with Definition 2.23 in \cite{lisuriart} (Definition 2.30 in \cite{lisuri}).
The result then follows from Lemmas 2.25  and 2.26 in \cite{lisuriart} (Lemmas 2.32 and 2.33 in \cite{lisuri}), applying Lemma 2.38 in \cite{lisuriart} (Lemma 2.43 in \cite{lisuri}) again.
\end{proof}
 
\begin{lemma}\label{morleyloytyy}
Let $A$ be a finite set and $\kappa$ a cardinal such that $\kappa=\textrm{cf}(\kappa)>\omega$.
For every sequence $(a_i)_{i<\kappa}$, there is some countable model $\A \supset A$ and some
$X \subset \kappa$ cofinal so that $(a_i)_{i \in X}$ is Morley over $\A$.
\end{lemma}  

\begin{proof}
By Lemma 2.38 in \cite{lisuriart} (Lemma 2.43 in \cite{lisuri}), our definition of a Morley sequence is equivalent with Definition 2.23 in \cite{lisuriart} (Definition 2.30 in \cite{lisuri}).
The result then follows from Lemma 2.24 in \cite{lisuriart} (Lemma 2.31 in \cite{lisuri}), since we may always choose the model $\A$ obtained to be countable (just take a suitable submodel).
\end{proof} 
 
\begin{lemma}\label{tyyppilemma}
Let $(a_i)_{i<\o_1}$ and $(b_i)_{i<\o_1}$ be two sequences independent and strongly indiscernible over $A$.
Suppose $t^g(a_0/A)=t^g(b_0/A)$.
Then, for any $i_1 < \cdots <i_n <\o_1$ and $j_1< \cdots < j_n<\o_1$, we have
$$t^g(a_{i_1}, \ldots, a_{i_n}/A)=t^g(b_{j_1}, \ldots, b_{j_n}/A).$$
\end{lemma}

\begin{proof}
As in the first-order case, using results from \cite{lisuriart}, section 2 (see also \cite{lisuri}, chapter 2).
%(Automorfismeilla pyšrittŠmŠllŠ; kŠyttŠmŠllŠ sopivaa automorfismia voidaan olettaa, ettŠ jonoissa sama tyyppi.
\end{proof} 
   
We still need the following lemma that will be applied in the proofs for (8) and (9).

\begin{lemma}\label{ylimalli}
Let $\mathfrak{V}$ be a saturated elementary extension of $V$, let $K$ be the kernel of exp in $V$ and $K^*$ the kernel of exp in $\mathfrak{V}$. 
If $X \subset \mathfrak{V}$ is a set such that $\textrm{span}(X) \cap K^* \subseteq K$, then
there is a set $X_V \subseteq V$ isomorphic to $X$.
\end{lemma}   

\begin{proof}
Denote $\mathfrak{F}=\textrm{exp}(\mathfrak{V})$ and $X_0=\textrm{span}(X \cup \textrm{log}(\mathbb{Q}))$.
Let $T$ be the theory of covers.
We will enlarge $X_0$ to a set $X' \supset X_0$ such that  $X'$ is a model for the theory $T+(K\cong\mathbb{Z})$.
This theory is categorical,  
so $X'$ is isomorphic to some elementary submodel of our monster model $V$, and thus we can find a set isomorphic to $X \subseteq X'$ already in $V$.
 
We enlarge $X_0$ as follows. 
It is easy to see that $\textrm{exp}(X_0)$ is closed under multiplication and inversion and that if for some $a \in \mathfrak{F}$, $a^{\frac{m}{n}} \in \textrm{exp}(X_0)$ for some choice of the $n$:th root $a^\frac{1}{n}$, where $\frac{m}{n} \in \mathbb{Q}$, then $a \in \textrm{exp}(X_0)$. 
Let $A=\textrm{acl}(\textrm{exp}(X_0))$.
Pick some $d_0 \in A \setminus \textrm{exp}(X_0)$, and let $x_0 \in \mathfrak{V}$ be such that $\textrm{exp}(x_0)=d_0$.
Denote $X_1=\textrm{span}(X_0 \cup \{x_0\})$.
We claim that $X_1 \cap K^{*} \subseteq K$.
Suppose not.
Let $Y \subset X_0$ be a finite set such that $\sum_{v \in Y} q_v v+q x_0 \in K^{*} \setminus K$ for some 
$q_v, q \in \mathbb{Q}$.
Since $X_0 \cap K^*=K$, we must have $q \neq 0$.
This gives us  
$$d_0^q=\prod_{v \in Y} (\textrm{exp}(v))^{-q_v}$$ (for some suitable choices of the roots in question).
But now $d_0 \in \textrm{exp}(X_0)$ which is against our assumptions.  

Repeating the process (taking unions at limit steps), we eventually get a set $X' \subset \mathfrak{V}$ such that $\textrm{exp}(X')=A$ and $X' \cap K^{*}=K$.
\end{proof}
   
We are now ready to prove (8).   
   
\begin{theorem}\label{infamalg}
Let $(a_i:i\in I)$ be independent and strongly indiscernible over $b$, with $I$ infinite. 
Suppose $(a_i':i\in I)$ is strongly indiscernible over $b'$, and $a_ib \to a_i'b'$ for each $i \in I$.
Further suppose $\textrm{rk}(b \to b') \le 1$.
Then, $(ba_i:i \in I) \to (b'a_i':i\in I)$.
\end{theorem}

\begin{proof} 
Since the sequences are strongly indiscernible, we may without loss assume that $I=\omega_1$.
We will find a sequence $(c_i)_{i<\omega_1}$, such that $(bc_i : i < \omega_1) \to (b'a_i': i < \omega_1)$ and $t^g(a_{i1}, \ldots, a_{i_n}/b)=t^g(c_{i1}, \ldots, c_{i_n}/b)$ for all $i_1, \ldots, i_n <\omega_1$, and the claim will follow.

By Theorem \ref{amalg} and induction, there exist elements $a_i^*$, $i<\omega$, independent over $b$ such that $\textrm{t}^g(a_i^*/b)=\textrm{t}^g(a_i/b)$ and $b(a_i^*)_{i<\omega} \to b'(a_i')_{i<\omega}$.
Then, also $(\textrm{exp}(a_i^{*}))_{i <\omega}$ are independent over $\textrm{exp}(b)$.
Moreover, we have for all $i,j \in \omega$,  that $\textrm{tp}(\textrm{exp}(a_i^*)/\textrm{exp}(b))=\textrm{tp}(\textrm{exp}(a_j^*)/\textrm{exp}(b))$ (for complete first-order types).
Since algebraically closed fields are $\omega$-stable, there are only finitely many free extensions for each complete type.
Thus, there is an infinite $I_0  \subset \omega$ so that $(\textrm{exp}(a_i^{*}))_{i \in I_0}$ are indiscernible (in the field language) over $\textrm{exp}(b)$.
 
Let $\theta$ be a theory consisting of the first-order formulae that express the following (note that we have added the elements of $\textrm{cl}(\emptyset)=\textrm{log }(\overline{\mathbb{Q}})$ to our language):
\begin{itemize}
\item The sequence $(\textrm{exp}(x_i))_{i<\omega_1}$ satisfies the same first-order formulae over $\overline{\mathbb{Q}} \cup \textrm{exp}(b)$ as the sequence $(\textrm{exp}(a_i^{*}))_{i \in I_0}$ (note that this is possible as the latter sequence is indiscernible over $ \textrm{exp}(b)$);
\item For each $i<\omega_1$, $x_i$ has the same complete first-order type over $\textrm{cl}(\emptyset) \cup \{b\}$ as $a_j^*$ for some (and thus every) $j \in I_0$, i.e. $\textrm{tp}(x_i/\textrm{cl}(\emptyset) b)=\textrm{tp}(a_j^*/\textrm{cl}(\emptyset) b)$;
\item For each $n$ and each positive, quantifier free first-order formula $\phi$ such that $\neg \phi(b', a_1', \ldots, a_n')$ holds, the theory $\theta$ contains the formula $\neg \phi(b, x_{j_1}, \ldots, x_{j_n})$ for all $j_1< \ldots <j_n<\omega_1$;
\item For any $n$ and any $i_1<\ldots<i_n<\omega_1$, it holds that if $\textrm{exp}(q_1 x_{i_1}+ \ldots + q_n x_{i_n}+c)=1$, where $q_1, \ldots, q_n \in \mathbb{Q}$ and $c$ is some linear combination of elements in $\textrm{log}(\textrm{exp}(\overline{\mathbb{Q}}))\cup b$, then 
$$q_1 x_{i_1}+ \ldots + q_n x_{i_n}+c=q_1 a_1^*+ \ldots + q_n a_n^*+c.$$ (Note that this can be expressed since $q_1 a_1^*+ \ldots + q_n a_n^*+c \in K$, $K \subset \textrm{log}(\textrm{exp}(\overline{\mathbb{Q}}))$, and $(\textrm{exp}(x_i))_{i<\omega_1}$ satisfies the same first-order formulae  as the sequence $(\textrm{exp}(a_i^{*}))_{i \in I_0}$.)
\end{itemize}
This theory is consistent as every finite fragment is realized by the sequence $(a_i^{*})_{i \in I_0}$.
Thus, in a saturated elementary extension $\mathfrak{V}$ of  $V$, we find a sequence $(c_i)_{i<\omega_1}$ realizing $\theta$
(note that since $\mathfrak{V}$ is saturated, it is not a model of  $T+(K=\mathbb{Z})$).

Denote by $K$ the kernel of exp in $V$ and by $K^{*}$ the kernel of exp in  $\mathfrak{V}$. 
Since the $c_i$ satisfy the theory $\theta$, we have $\textrm{span}(c_i)_{i<\omega_1} \cap K^* \subseteq K$.
By Lemma \ref{ylimalli}, we may assume that $c_i \in V$ for each $i<\omega_1$.
 
By locality of independence and the choice of the sequence $(c_i)_{i<\o_1}$, the sequence $(\textrm{exp}(c_i))_{i<\omega_1}$ is independent over $\textrm{exp}(b)$ since $(\textrm{exp}(a_i^{*}))_{i \in I_0}$ is.
As all dimensions can be calculated on the field, the sequence $(c_i)_{i<\omega_1}$ is independent over $b$.
By Lemmas \ref{morleyloytyy} and \ref{morleyindisc}, there is some uncountable $J \subset \omega_1$ and some countable model $\A$ such that $b \in \A$ and $(c_i)_{i \in J}$ is Morley and thus strongly indiscernible over $\A$.
Relabel the indices so that from now  on $(c_i)_{i<\omega_1}$ stands for $(c_i)_{i \in J}$.
Since $(a_i')_{i<\omega_1}$ is strongly indiscernible over $b'$, we have $(bc_i: i<\o_1) \to (b'a_i' : i<\o_1).$
We have chosen the $c_i$ so that the sequence is independent over $b$, and by Remark \ref{le31}, $\textrm{tp}(c_0/b)=\textrm{tp}(a_0^*/b)$ implies $t^g(c_0/b)=t^g(a_0^*/b)=t^g(a_0/b).$
Thus, by Lemma \ref{tyyppilemma}, $t^g(a_{i_1}, \ldots, a_{i_n}/b)=t^g(c_{i_1}, \ldots, c_{i_n}/b)$ for all $i_1, \ldots i_n \in \omega_1$, so we are done.
\end{proof}

For (9), we still need the following lemma.
The Zariski geometry version was presented originally in \cite{HrZi} (Lemma 4.13).  
 
\begin{lemma}\label{dimthmspec}
Let $a=(a_1, \ldots, a_n), a''=(a_1'', \ldots, a_n'') \in V^n$, $a \to a''$, and suppose $a_1 \neq a_2$, $a_1''=a_2''$.
Then there exists $a'=(a_1', \ldots, a_n') \in V^n$ such that $a_1'=a_2'$, $a \to a' \to a''$, and $\textrm{dim}(a)-\textrm{dim}(a')=1$.
\end{lemma}

\begin{proof}
Denote $\Delta_{12}^n=\{(v_1, \ldots, v_n) \in V^n \, | \, v_1=v_2\}$.
Let $C$ be the locus of $a$.
Then, $a''$ must lie on some irreducible component $D$ of $C \cap \Delta^n_{12}$.
By Theorem \ref{dimthm}, 
$\textrm{dim}(X) \ge \textrm{dim}(C)-1.$
Since $a_1 \neq a_2$, we have $C \cap \Delta^n_{12} \subsetneq C$, so $\textrm{dim}(D)= \textrm{dim}(C)-1.$
 
By Lemma \ref{dimthmlemma}, $\textrm{exp}(D)$ is an irreducible component of $\textrm{exp}(C) \cap \Delta^n_{12}$ (abusing notation, we use $ \Delta^n_{12}$ to denote the diagonal also on the field).
By Corollary \ref{olennainen}, $D$ is an irreducible component of $\textrm{log}(\textrm{exp}(C) \cap \Delta_{12}^n)$.
The variety $\textrm{exp}(C) \cap \Delta_{12}^n$ is definable over the empty set, and thus also $X$ is definable over the empty set as a variety.
By Remarks \ref{irredykskas} and \ref{kuntapqf} and Theorem \ref{komponentit}, $D$ is PQF-closed over the empty set.
 
Choose $a'$ to be a generic point of $D$.
Then $a'$ is as wanted. 
\end{proof} 

And finally, we prove (9):
 
\begin{theorem}\label{dimthmspec2}
 Let  $a_i, b_i \in V$ with $i < \kappa$, such that $a_0 \neq a_1$ and $b_0=b_1$.
Denote by $K$ the kernel of exp in $V$.
Suppose $(a_i)_{i<\kappa} \to (b_i)_{i <\kappa}$ is a specialization.
Assume there is some unbounded and directed $S \subset \mathcal{P}_{<\omega}(\kappa)$ satisfying the following conditions:
\begin{enumerate}[(i)]
\item  $0,1 \in X$ for all $X \in S$;
\item For all $X,Y \in S$ such that $X \subseteq Y$,  and for all sequences $(c_i)_{i \in Y}$ from $V$, the following holds: 
If $c_0=c_1$, $ (a_i)_{i \in Y} \to (c_i)_{i \in Y} \to (b_i)_{i \in Y},$
and $\textrm{rk}((a_i)_{i \in Y} \to (c_i)_{i \in Y}) \le 1$, then $\textrm{rk}((a_i)_{i \in X} \to (c_i)_{i \in X}) \le 1$.
\end{enumerate}

\vspace{3mm}  
Then, there are $(c_i)_{i<\kappa}$ such that   
$$(a_i)_{i \in \kappa} \to (c_i)_{i \in \kappa} \to (b_i)_{i \in \kappa},$$
$c_0=c_1$ and $\textrm{rk}((a_i)_{i \in X} \to (c_i)_{i \in X}) \le 1$ for all $X \in S$.
\end{theorem}
 
\begin{proof}
Let  $a_i, b_i \in V$ with $i < \kappa$, such that $a_0 \neq a_1$ and $b_0=b_1$.
Suppose $(a_i)_{i<\kappa} \to (b_i)_{i <\kappa}$ is a specialization.
If $S \subset \mathcal{P}_{<\omega}(\kappa)$ is such that $0,1 \in X$ for all $X \in S$, then, by Lemma \ref{dimthmspec}, for each $X \in S$, there is some sequence $(c_i)_{i \in X} \in V$ so that 
$$(a_i)_{i \in X} \to (c_i)_{i \in X} \to (b_i)_{i \in X},$$
$c_0=c_1$ and $\textrm{rk}((a_i)_{i \in X} \to (c_i)_{i \in X}) \le 1$.

Suppose now that $S$ is unbounded and directed and satisfies condition (ii) from the theorem.
By Compactness, there is a saturated elementary extension $\mathfrak{V}$ of $V$ and elements $c_i \in \mathfrak{V}$ for $i < \kappa$ such that $(a_i)_{i<\kappa} \to (c_i)_{i<\kappa} \to (b_i)_{i <\kappa}$, $c_0=c_1$ and $\textrm{rk}((a_i)_{i \in X} \to (c_i)_{i \in X}) \le 1$ for all $X \in S$.
Denote by $K$ the kernel of exp in $V$ and by $K^{*}$ the kernel of exp in  $\mathfrak{V}$.

We will show, using Compactness, that we may choose the $c_i$ so that $\textrm{span}(c_i)_{i < \kappa} \cap K^{*} \subset K$.
Then, by Lemma \ref{ylimalli}, we may assume that the sequence $(c_i)_{i<\kappa}$ is in $V$ and the theorem will follow.
  
Let $J \subset \kappa$ be finite.  
By Lemma \ref{dimthmspec}, there are $c_i' \in V$ such that $(a_i)_{i \in J} \to (c_i')_{i \in J} \to (b_i)_{i \in J}$
and $\textrm{rk}((a_i)_{i \in X} \to \textrm{rk}(c_i')_{i \in X}) \le 1$ for all $X \in S \cap \mathcal{P}(J)$.
If for some $I_0 \subseteq J$, we have $\textrm{exp}(\sum_{i \in I_0} q_i c_i')=1$, where $q_i \in \mathbb{Q}$, 
then, since the $c_i'$ are in $V$, we have $\sum_{i \in I_0} q_i c_i'=k$ for some $k \in K$.
As there is a specialization from the $c_i'$ to the $b_i$, we must have $\sum_{i \in I_0} q_i b_i=k$.
Thus, for any finite $J \subset \kappa$, we may choose the sequence $(c_i)_{i \in J}$ so that whenever $I_0 \subseteq J$ and  $\textrm{exp}(\sum_{i \in I_0} q_i c_i)=1$, then $\sum_{i \in I_0} q_i c_i=\sum_{i \in I_0} q_i b_i$.
Hence, by Compactness, we may choose the $c_i$ for $i<\kappa$ so that whenever $I_0 \subset \kappa$ is finite and $\textrm{exp}(\sum_{i \in I_0} q_i c_i)=1$, then $\sum_{i \in I_0} q_i c_i=\sum_{i \in J} q_i b_i$.
In particular, $\textrm{span}(c_i)_{i \in I} \cap K^{*} \subset K$.
\end{proof}
    
\subsection{Curves on the cover}

\begin{definition}
We say that an irreducible, one-dimensional PQF-closed set  $D$ on $V^n$ is a \emph{curve} on $V^n$.
Note that if $D$ is a curve on $V^n$, then $\textrm{exp}(D)$ is an algebraic curve on $F^n$.

For each $m$, define the \emph{closed sets} on  $D^m$  to be the restrictions of the PQF-closed sets on $V^{mn}$.
Again, we say that a closed set is \emph{irreducible} if it cannot be written as a union of two proper closed subsets.
\end{definition}

We first note that if $D \subset V^n$ is a curve, then each point $x \in D^m$ is also a point of $V^{nm}$ and the locus of $x$ on $D^m$ coincides with the locus of $x$ on $V^{nm}$.
It follows that also ranks coincide and that a map on $D$ is a specialization if and only if the corresponding map on $V$ is one.
In fact, if we assume $\textrm{exp}(D)$ is a smooth algebraic curve, then the theory developed for the cover nicely descends to $D$, and the results can be proved using basically the same arguments as in the case of the cover.
The smoothness is needed because it guarantees that the dimension theorem will hold (see \cite{Marker}).
 If $\textrm{exp}(D)$ is not smooth, then we might run into trouble proving (9) since the dimension theorem does not hold on an arbitrary curve.
The other axioms for a Zariski-like structure can be proved, however.

\end{document}